\documentclass [leqno,11pt]{amsart}    
\usepackage{amssymb,amsmath,latexsym,amsfonts,amsbsy,amsthm,mathtools,graphicx,color}    
\usepackage{hyperref}
\usepackage{scalerel,stackengine}
\stackMath
\newcommand\reallywidehat[1]{%
\savestack{\tmpbox}{\stretchto{%
  \scaleto{%
    \scalerel*[\widthof{\ensuremath{#1}}]{\kern-.6pt\bigwedge\kern-.6pt}%
    {\rule[-\textheight/2]{1ex}{\textheight}}%WIDTH-LIMITED BIG WEDGE
  }{\textheight}% 
}{0.5ex}}%
\stackon[1pt]{#1}{\tmpbox}%
}
\mathtoolsset{showonlyrefs=true}

\definecolor{MyDarkBlue}{rgb}{0,0.08,0.45}
\definecolor{yellow}{rgb}{0.99,0.99,0.70}
\definecolor{myback}{RGB}{204,232,207}
\definecolor{white}{rgb}{1.0,1.0,1.0}                                          
\definecolor{black}{rgb}{0.00,0.00,0.00}
\usepackage{cancel}
\usepackage[dvipsnames]{xcolor}
\numberwithin{equation}{section}

\allowdisplaybreaks
\let\al=\alpha
\let\b=\beta

\let\d=\delta

\let\la=\lambda

\let\f=\frac

\let\na=\nabla
\let\th=\theta
\let\pa=\partial

\def\bbT{\mathbb T}
\def\no{\noindent}

\def\bbT{\mathbb{T}}

\def\eqdefa{\buildrel\hbox{\footnotesize def}\over =}

\usepackage{accents}

\newcommand{\beq}{\begin{equation}}
\newcommand{\eeq}{\end{equation}}
\newcommand{\ben}{\begin{eqnarray}}
\newcommand{\een}{\end{eqnarray}}
\newcommand{\beno}{\begin{eqnarray*}}
\newcommand{\eeno}{\end{eqnarray*}}
\definecolor{schrift}{RGB}{0,73,114}

\newtheorem{theorem}{Theorem}[section]

\newtheorem{lemma}[theorem]{Lemma}
\newtheorem{proposition}[theorem]{Proposition}

\begin{document}

\title[Enhanced dissipation and the non-local enhancement]{Metastability for the dissipative quasi-geostrophic equation and the non-local enhancement}
\author{Hui Li}
\address{Department of Mathematics, New York University Abu Dhabi, Saadiyat Island, P.O. Box 129188, Abu Dhabi, United Arab Emirates, Department of Mathematics, Zhejiang University, Hangzhou 310027, China.}
\email{lihui@nyu.edu,lihui92@zju.edu.cn}
\author{Weiren Zhao}
\address{Department of Mathematics, New York University Abu Dhabi, Saadiyat Island, P.O. Box 129188, Abu Dhabi, United Arab Emirates.}
\email{zjzjzwr@126.com, wz19@nyu.edu}

\date{}

\begin{abstract}
In this paper, we study the metastability for the 2-D linearized dissipative quasi-geostrophic equation with small viscosity $\nu$ around the quasi steady state $\theta_{sin}=e^{-\nu t}\sin y$. We proved the linear enhanced dissipation and obtained the dissipation rate. Moreover, the new non-local enhancement phenomenon was discovered and discussed. Precisely we showed that the non-local term re-enhances the enhanced diffusion effect by the shear-diffusion mechanism. 
\end{abstract}

\maketitle

\section{Introduction} 
In this paper, we study the 2D dissipative quasi-geostrophic (QG) equation on the torus $\mathbb{T}_{\delta}^2=\big\{(x,y): x\in \bbT_{{2\pi}{\delta}}, y\in \bbT_{2\pi}\big\}$ with $0<\delta< 1$:
\begin{equation}\label{eq:QG}
  \left\{
  	\begin{array}{l}
  		\pa_t\Theta+\nu (-\Delta)^{s} \Theta+V\cdot\nabla \Theta=0,\\
V=(-R_2,R_1)\Theta.
  	\end{array}
  \right.
\end{equation}
Here $s\in(0,1]$, $\nu>0$ is the dissipative coefficient,  $\Theta: \mathbb{T}_{\d}^2\to \mathbb{R}$ is a real-valued scalar function represents the potential temperature, $V$ is the fluid velocity and $R_1=\pa_x(-\Delta)^{\frac{1}{2}}$ and $R_2=\pa_y(-\Delta)^{\frac{1}{2}}$ are the Riesz transforms in $\mathbb{T}_{\d}^2$. The cases $s>\frac{1}{2}$, $s=\frac{1}{2}$, and $s<\frac{1}{2}$ are called sub-critical, critical, and super-critical respectively.

%In this paper we consider the critical case and fix $\alpha=\frac12$. 

%For the inviscid quasi-geostrophic equation, there exists a class of steady solution 
%\begin{align*}
%\hat \theta_s=f(y),\quad V_s=(-H(f),0)
%\end{align*}
%where $H$ is the usual Hilbert transform in $\mathbb{T}_{2\pi}$. 
%
%For the critical quasi-geostrophic equation, it is easy to see that 
%\begin{align*}
%\hat \theta_s=e^{-\nu tH\pa_y}f(y),\quad V_s=(-e^{-\nu tH\pa_y}H(f),0)
%\end{align*}
%where $H\pa_y=|\pa_y|$. 
%
%For the special case $f(y)=\sin y$, 
It is easy to see that the system \eqref{eq:QG} has a quasi steady state
\begin{align*}
\theta_{sin}=e^{-\nu t}\sin y,\quad V_{sin}=(e^{-\nu t}\cos y,0).
\end{align*} 

To understand the long time behavior of the solution with initial data $\Theta_{in}$ close to $\sin y$, it is natural to introduce the perturbation $\Theta=\theta_{sin}+\theta$ and $V=V_{sin}+U$. Then $(\theta, U)$ solves the following system:
\begin{equation}\label{eq:QG-2}
  \left\{
  	\begin{array}{l}
  		\pa_t\theta+\mathcal{L}_{\nu,s}(t)\theta=-U\cdot\na \theta\\
    U=(-R_2,R_1)\theta,
  	\end{array}
  \right.
\end{equation}
where
\begin{align}\label{Lnut}
	\mathcal{L}_{\nu,s}(t)=\nu (-\Delta)^{s}+e^{-\nu t}\cos y\pa_x(1-(-\Delta)^{-\frac12}).
\end{align}
From which one can derive the linearized QG equation:
\begin{align}\label{eq:L-QG}
\left\{\begin{aligned}
&	\pa_t\theta+\mathcal{L}_{\nu,s}(t)\theta=0\\
& \th|_{t=0}=\th_{in}. 
\end{aligned}\right. 
\end{align}
Taking Fourier transform in $x$ of \eqref{eq:L-QG}, we get that for $\al\neq 0$
\begin{equation}\label{eq:L-QG-al}
  \left\{
    \begin{array}{l}
      \pa_t\hat \theta+\nu(-\Delta_{\al})^{s}\hat \theta+i\al e^{-\nu t}\cos y \big(1-(-\Delta_{\al})^{-\frac{1}{2}}\big)\hat \theta=0,\\
      \hat\theta|_{t=0}=\hat \theta_{in}.
    \end{array}
  \right.
\end{equation}
Here $|\al|\ge c_0>1$ is the wave number and $\Delta_{\al}=\pa_{yy}-\al^2$. We also denote the linear operator by
\begin{align*}
   \mathcal{L}_{\al,\nu, s}=\nu(-\Delta_{\al})^{s}+i\al e^{-\nu t}\cos y \big(1-(-\Delta_{\al})^{-\frac{1}{2}}\big).
\end{align*} 
In this paper, we mainly study the long time behavior of the solutions to \eqref{eq:L-QG-al}.
\subsection{Historical comments}
The 2-D dissipative quasi-geostrophic (QG) equation is derived from the study of geophysical fluid dynamics \cite{Ped1987} and attracted a lot of attention from mathematicians as it it serves as a lower dimensional model of the 3-D Navier-Stokes equations \cite{CMT1994}. The existence of global regular solutions for the sub-critical QG equation was given by Constantin-Wu \cite{CW1999} and Resnick \cite{Re1995}. For the critical case, the global regularity was proved independently by Kiselev-Nazarov-Volberg \cite{KNV2007} and Caffarelli Vasseur \cite{CV2010}. While for the super-critical case, whither solutions stay regular or can blow up is still open, and Dabkowski-Kiselev-Silvestre-Vicol \cite{DKSV2014} proved global regularity for the slightly super-critical QG equation. For more results about QG equation, we refer to \cite{Co2002,Ju2004,KRYX2016}.

Metastability is common in viscous fluid, which is a quasi-stable state of a dynamical system other than the system's state of least energy. A famous example is the Kolmogorov flow for the 2-D incompressible Navier-Stokes equations. The previous studies \cite{BW,Cou,LX,MS,WZZ3} have shown that the solution with initial data closed to $(\cos y,0)$ will approach to the so called bar states in a much shorter time than the diffusive time $O(\f{1}{\nu})$. Then they dominate the dynamics for very long time intervals. Afterwards they decay on the diffusive time scale. See \cite{BW-B} for a similar phenomena for Burgers equation with small viscosity. 

A reasons for metastability is the so called enhanced dissipation. It is sometimes referred to by modern authors as the `shear-diffusion mechanism'. This decay rate is much faster than the diffusive decay of $e^{-\nu t}$. The mechanism leading to the enhanced dissipation is due to the mixing. Generally speaking, the sheared velocity sends information to higher frequency, the diffusion term `kill' the information in the higher frequency. The degeneracy rate of sheared velocity is corresponding to the lowest speed how the information moves to higher frequency, which leads to the different enhanced dissipation rates \cite{DRM2021,BC2017,Coti2020,Wei2021}. Also the stronger diffusion gives stronger enhanced dissipation \cite{He2021}. We refer to some recent results for the enhanced dissipation phenomena of different flows in the fluid motions: Couette flow  \cite{BVW2018,CLWZ2018transition,Masmoudi_Zhao_2019,MZ2020enhanced}, Poiseuille flow \cite{CEW2020,DL2020} and other flows \cite{GRENIER2020108339,LWZ2020}. 

In this paper we proved the enhanced dissipation for the linearized dissipative QG equations. {\bf Moreover we discovered and studied the non-local enhancement phenomenon, which means that, the non-local term re-enhances the enhanced dissipation. }

\subsection{Main results}

Our first result is the linear enhanced dissipation for the critical case: 
\begin{theorem}[Critical case]\label{Thm-cri}Suppose that $\theta$ solves \eqref{eq:L-QG} with $s=\f12$. Then there exists $\nu_0\le1,\, c>0$, $C>1$ such that for $0<\nu<\nu_0$ and $0\leq t\leq C^{-1}\nu^{-1}$, it holds that 
\begin{align*}
  \|\theta_{\neq}(t)\|_{L^2(\mathbb{T}_{\delta}^2)}\leq Ce^{-c\nu^{\f35}t}\|\theta_{in}\|_{L^2(\mathbb{T}_{\delta}^2)}
\end{align*}
where $\theta_{\neq}(t,x,y)=\theta(t,x,y)-\f{1}{2\pi \delta}\int_{\mathbb{T}_{2\pi \delta}}\theta(t,x,y)dy$. 
\end{theorem}
In order to show the non-local enhancement phenomenon, we also studied a toy model and proved the following result. 
\begin{theorem}[Toy model]\label{thm-toy}Suppose that $\theta$ solves 
\begin{equation}
  \left\{
    \begin{array}{l}
      \pa_t \th+\cos y\pa_x\th+\nu (-\Delta)^{\f12}\th=0\\
      \th|_{t=0}=\th_{in}.
    \end{array}
  \right.
\end{equation}
Then there exists $\nu_0\le 1,\, c>0$, $C>1$ such that for $0<\nu<\nu_0$, it holds that 
\begin{align*}
  \|\theta_{\neq}(t)\|_{L^2(\mathbb{T}_{\delta}^2)}\leq Ce^{-c\nu^{\f23}t}\|\theta_{in}\|_{L^2(\mathbb{T}_{\delta}^2)}
\end{align*}
where $\theta_{\neq}(t,x,y)=\theta(t,x,y)-\f{1}{2\pi \delta}\int_{\mathbb{T}_{2\pi \delta}}\theta(t,x,y)dy$. 
\end{theorem}
We also study the linearized subcritical QG which is similar to the linearized Navier-Stokes equation. We find that the different non-local terms will affect the enhanced dissipation rate. 
\begin{theorem}[Subcritical case]\label{thm-sub}Suppose that $\theta$ solves \eqref{eq:L-QG} with $s=1$. Then there exists $\nu_0\le 1,\, c>0$, $C>1$ such that for $0<\nu<\nu_0$ and $0\leq t\leq C^{-1}\nu^{-1}$, it holds that 
\begin{align*}
  \|\theta_{\neq}(t)\|_{L^2(\mathbb{T}_{\delta}^2)}\leq Ce^{-c\nu^{\f37}t}\|\theta_{in}\|_{L^2(\mathbb{T}_{\delta}^2)}
\end{align*}
where $\theta_{\neq}(t,x,y)=\theta(t,x,y)-\f{1}{2\pi \delta}\int_{\mathbb{T}_{2\pi \delta}}\theta(t,x,y)dy$. 
\end{theorem}

\subsection{Discussion}
Let us first introduce the generalized linear operator
\begin{align}
	\mathcal{L}_{\al,\nu, s,\tilde s}=\nu (-\Delta_{\al})^{ s}+i\al\cos y(1-(-\Delta_{\al})^{-\tilde s})
\end{align}
and
\begin{align}
	\mathcal{L}^S_{\al,\nu, s}=\nu (-\Delta_{\al})^{ s}+i\al \cos y
\end{align}
with $ s,\tilde s>0$. Let 
\begin{align*}
	&B_{\al,\tilde s}=\cos y \big(1-(-\Delta_{\al})^{-\tilde s}\big),\qquad\qquad\qquad\quad\ \ \,  B^S=\cos y,\\
	&A_{\al,\tilde s}=-[\pa_y, B_{\al,\tilde s}]=\sin y \big(1-(-\Delta_{\al})^{-\tilde s}\big),\qquad A^S=-[\pa_y,B^S]=\sin y.
\end{align*}

We summary the results in the following table. 
\begin{table}[!htbp]
\centering
\begin{tabular}{|c|c|c|c|c|c|}
\hline  
&Operator & Parameters & Dissipation rate &Equations & Reference\\
\hline  
(1)&$\mathcal{L}_{\al,\nu, s,\tilde s}$ & $ s=\tilde s=1$  &$e^{-c\nu^{\f12}t}$& Linearized Navier-Stokes &\cite{IMM,WZ2019,WZZ3}\\
\hline 
(2)&$\mathcal{L}^S_{\al,\nu, s}$ & $ s=1$ & $e^{-c\nu^{\f12}t}$& Transport diffusion&\cite{BW,He2021}\\
\hline
(3)&$\mathcal{L}_{\al,\nu, s,\tilde s}$ & $ s=1$, $\tilde s=\f12$ &$e^{-c\nu^{\f37}t}$& Linearized sub-critical QG& this paper\\
\hline
(4)&$\mathcal{L}_{\al,\nu, s,\tilde s}$ & $ s=\tilde s=\f12$ &$e^{-c\nu^{\f35}t}$& Linearized critical QG&this paper\\
\hline
(5)&$\mathcal{L}^S_{\al,\nu, s}$ & $ s=\f12$ &$e^{-c\nu^{\f23}t}$& Transport fractional diffusion& this paper\\
\hline
\end{tabular}
\end{table}

Comparing the results listed above, we have the following observations:
\begin{itemize}
\item The non-local term in Case $(1)$ does not affect the dissipation rate. 
\item The result $(3)$ in comparison with result $(1)$ and $(2)$, indicates that the non-local term accelerates the dissipation rate. Similar phenomenon happens in the fractional diffusion case, which can be found in result (4) and (5). Actually the non-local term is a compact perturbation to the non-self-adjoint operator $\mathcal{L}^S_{\al,\nu, s}$. However, the non-self-adjoint operator is sensitive to the compact perturbations. 
\item The stronger diffusion term leads to the faster decay rate. Indeed by modifying our proof, one can obtain that for $0< s\leq 1$, 
\ben
\left\|e^{-t\mathcal{L}_{\al,\nu, s,\frac{1}{2}}}\right\|_{L^2\to L^2}\lesssim e^{-c\nu^{\frac{3}{3+4 s}}t}, \quad |\al| >1,
\een
and 
\ben
\left\|e^{-t\mathcal{L}^S_{\al,\nu, s}}\right\|_{L^2\to L^2}\lesssim e^{-c\nu^{\frac{2}{2+2 s}}t}, \quad \al\neq 0.
\een
As for the second result in comparison with which in \cite{He2021}, we can remove the logarithm loss by using suitable time weight. 
\item Our method works well for the transport fractional diffusion equation with general sheared velocity $(w(y),0)$:
\beno
\pa_t \th+w(y)\pa_x\th+\nu (-\Delta)^{ s}\th=0,\quad 0<s\leq 1.
\eeno
One may easily modify our proof and obtain the following result:

{\it Suppose that $w\in C^4(\mathbb{T})$ has only non-degenerate critical points, then it holds that
\beno
\|\th_{\neq}(t)\|_{L^2}\leq Ce^{-c\nu^{\frac{2}{2+2 s}}t}\|\th_{in}\|_{L^2}. 
\eeno
}
\end{itemize}

The paper is organized as follows: In section 2, we give a formal explanation of the re-enhancing effect of the non-local term and show the main idea of the proof for the linearized QG equation in critical case. In section 3, we introduce some important auxiliary lemmas and give the energy estimates. In section 4, we establish the decay estimate and complete the proof of Theorem \ref{Thm-cri}. In section 5, we discuss the sub-critical QG and show the influence of different diffusion term. In section 6, we study the toy model and clarify the non-local enhancement.

{\bf Notations}: 
%Let us specify the notations to be used throughout the paper. We denote by $A \lesssim B$ an estimate of the form $A \leq C B$ and by $A \sim B$ an estimate of the form $C^{-1} B \leq A \leq C B,$ where $C$ is a constant. 
Given a function $f(x,y),$ we denote its Fourier transform in $x$-variable as 
\beno
\hat f(\alpha, y)=\frac{1}{2\pi\d}\int_{\mathbb{T}_{2\pi \d}}e^{-ix\al}f(x,y)dx,
\eeno 
and its Fourier transform in $(x,y)$ as 
\beno
\tilde f(\alpha, \beta)=\mathcal F_y\big(\hat f(\alpha, \cdot)\big)=\frac{1}{4\pi^2\d}\int_{\mathbb{T}_{\delta}^2}e^{-ix\al-iy\beta}f(x,y)dxdy,
\eeno 
where $\alpha$ and $\beta$ are the wave numbers.  When no confusion can arise, we will short $\hat f(\alpha, y)$ and $\tilde f(\alpha, \beta)$ to $\hat f(y)$ and $\widetilde f(\beta)$.

%Comparing the results for $(1)$ and $(2)$, one can see that when $\tilde s=1$ the non-local term does not affect the dissipation rate. From the comparison of $(1)$, $(2)$ and $(3)$ as well as $(4)$ and $(5)$ we can see that the non-local term $(-\Delta_\al)^{-\frac{1}{2}}$ accelerates the dissipation rate. It is clear from $(3)$ and $(4)$ that the stronger diffusion term leads to the faster decay rate. For $\tilde s=\frac{1}{2}$ and $ s\in(0,1]$, we also derive that the dissipation rate correspondent to $\mathcal{L}_{\al,\nu, s,\frac{1}{2}}$ is $e^{-c\nu^{\frac{3}{3+4 s}}t}$.

%Now we try to explain the enhanced dissipation rate. 

%From the results of \cite{He2021}, it is natural to see that the enhanced dissipation rate of the system without non-local term is $e^{-c_0\nu^{\f23+}t}$. While with the non-local term the enhanced dissipation is stronger and of the rate $e^{-c_0\nu^{\f35}t}$. 

\section{Hypocoercivity and non-local enhancement}
In this section, we give the main idea of the proof for the linearized QG equation in critical case. 
By taking $s=\frac{1}{2}$ of \eqref{eq:L-QG-al}, we get that
\begin{align}\label{eq-ft-x}
	\pa_t\hat \theta+\nu(-\Delta_{\al})^{\frac{1}{2}}\hat \theta+i\al e^{-\nu t}\cos y \big(1-(-\Delta_{\al})^{-\frac{1}{2}}\big)\hat \theta=0.
\end{align}
%Here we use $\hat \theta(\alpha,y)$ to denote $\frac{1}{2\pi\d}\int_{\mathbb{T}_{2\pi\d}}\theta(x,y)e^{-ix\al}dx$.
Let
\begin{align}\label{eq-op-ab}
	A=\sin y \big(1-(-\Delta_{\al})^{-\frac{1}{2}}\big),\quad B=\cos y \big(1-(-\Delta_{\al})^{-\frac{1}{2}}\big),\quad \gamma(t)=\alpha e^{-\nu t},
\end{align}
and
\begin{align*}
	\mathcal L_\nu(\al,t)=\nu(-\Delta_{\al})^{\frac{1}{2}}+i\gamma(t)B.
\end{align*}
Then \eqref{eq-ft-x} can be written as
\begin{align*}
	\pa_t\hat \theta+\mathcal L_\nu \hat \theta=0.
\end{align*}
The non-self-adjoint operator $\mathcal L_\nu$ can be regarded as a sum of a self-adjoint operator $\nu(-\Delta)^{\f12}$ and a anti-self-adjoint operator $i\gamma(t)B$ defining on a weighted Hilbert space equipped with a special inner product 
\begin{align*}
	\langle u, w\rangle_*=\langle u, w-(-\Delta_{\al})^{-\frac{1}{2}}w\rangle,
\end{align*}
where 
\begin{align*}
	\langle u,w \rangle=\int_{\bbT} u(y)\overline{w(y)}dy, \quad 	\|u\|_{L^2}=\langle u,u\rangle^{\frac{1}{2}},
\end{align*}
are the ordinary $L^2$ inner product and $L^2$ norm. 

It is easy to check that the operators $A$ and $B$ are symmetric under this inner product,
\begin{align*}
	\langle u,A w\rangle_*=\langle Au, w\rangle_*,\quad \langle u,B w\rangle_*=\langle Bu, w\rangle_*.
\end{align*}
Let us also introduce the weighted $L^2$ norm $\|u\|_{*}=\langle u, u\rangle_*^{\frac{1}{2}}$. For $|\al|>1$, it is clear that
\begin{align*}
  (1-|\al|^{-1})\|u\|^2_{L^2}\le\|u\|_{*}^2\le\|u\|^2_{L^2}.
\end{align*}

\subsection{Hypocoercivity} 
It is natural to introduce energy functional based on the hypocoercivity method \cite{Villani2009}:
\begin{align*}
  e(t)=\|\hat{\th}\|_{*}^2+\la_1\left\|(-\Delta_{\al})^{\f14}\hat{\th}\right\|_{*}^2+\la_2\left\langle \big[(-\Delta_{\al})^{\f14},B\big]\hat{\th}, (-\Delta_{\al})^{\f14}\hat{\th}\right\rangle_{*}+\la_3\left\|\big[(-\Delta_{\al})^{\f14},B\big]\hat{\th}\right\|_{*}.
\end{align*}
However the energy functional does not work well in our problem. 

Let us introduce the modified energy function:
\begin{align*}
	\Phi(t)=E_0+a_1\nu^{2}t^2E_1+a_2\nu^{2}t^3\mathcal E_1+a_3\nu^{2}t^4\mathcal E_2,
\end{align*}
where 
\begin{align*}
	E_0=\|\hat \theta\|_*^2,\quad E_1=\|\pa_y\hat \theta\|_*^2,\quad \mathcal E_1=-\frac{\al}{|\al|}\Re\langle iA\hat \theta, \pa_y\hat \theta\rangle_*,\quad \mathcal E_2=\|\hat \theta\|_*^2-\|B\hat \theta\|_*^2.
\end{align*}

The key modifications are as follows. First of all, we introduce the time weighted norm which reduces the required regularity of the initial data and is widely used in studying the wellposedness theory in low regularity space \cite{CZZ2016,FK1964}. There are two types of terms in the energy functional. Heuristically, one may say that the `pure' terms $E_0$ and $E_1$ will mainly feel the influence of the symmetric part in $\mathcal L_\nu$, but that the `mixed' term $\mathcal E_1$ will mainly feel the influence of the antisymmetric part in $\mathcal L_\nu$. Here we modified $E_1$ and the corresponding `mixed' term $\mathcal E_1$, so that technically we can avoid treating the commutator of two non-local operators. The last modification is replacing the corresponding `pure' term $\|A\hat{\th}\|_{*}$ by $\mathcal E_2$ which is from the idea in \cite{WZ2019}. Actually $\mathcal E_2$ behaves as $\|A\hat{\th}\|_{*}^2$ and it is conserved for inviscid case $(\nu=0)$.  

\subsection{Formal deduction}
Next, we would like to give a formal mathematical explanation why the non-local term re-enhances the enhanced dissipation by the `shear-diffusion mechanism'. Let us simplify the energy functional by removing the time weight and define
\begin{align}
	E(t)=E_0+a_1\nu^{b_1}E_1+a_2\nu^{b_2}\mathcal E_1+a_3\nu^{b_3}\mathcal E_2.
\end{align}
Note that $\|\cdot\|_{*}$ behaves as the $L^2$ norm and $\mathcal E_2$ behaves as $\|A\hat{\th}\|_{*}^2$ (see Lemma \ref{lem-equivalence}). 

The commutator $A_{\al,\frac{1}{2}}=-[\pa_y,B_{\al,\frac{1}{2}}]=\sin y(1-(-\Delta_{\al})^{-\frac{1}{2}})$ plays an important role in the enhanced dissipation rate. 
Heuristically, due to the existence of nonlocal operator $(-\Delta_{\al})^{-\frac{1}{2}}$ in $A_{\al,\frac{1}{2}}$, we have $A_{\al,\frac{1}{2}}\approx |D|^{-\frac{1}{2}}$ and more generally $A_{\al,\tilde{s}}\approx |D|^{-\tilde{s}}$. While without the nonlocal term, formally $A^S\approx |D|^{-1}$. One may regard Lemma \ref{lem-equivalence} and Lemma \ref{lem-inter} as the evidences. For a more precise description of the commutator $A_{\al,\tilde{s}}$ or $A^S$, we conjecture as follows:

{\it
\no{\bf Conjecture:} Let $\la$ be the eigenvalue of $\mathcal{L}_{\al,\nu, s,\tilde s}$ with largest real part and $\hat\th_{\la}$ be the corresponding eigenfunction, i.e. $\mathcal{L}_{\al,\nu, s,\tilde s}\hat\th_{\la}=\la\hat\th_{\la}$. Then,
\begin{align*}
   C_{\al}^{-1}\|\hat\th_{\la}\|_{H^{-\tilde s}}\leq \|A_{\al,\tilde s}\hat\th_{\la}\|_{L^2}\leq C_{\al} \|\hat\th_{\la}\|_{H^{-\tilde s}},
 \end{align*} 
Let $\la$ be the eigenvalue of $\mathcal{L}^S_{\al,\nu, s}$ with largest real part and $\hat\th_{\la}$ be the corresponding eigenfunction, i.e. $\mathcal{L}^S_{\al,\nu, s}\hat\th_{\la}=\la\th_{\la}$. Then, 
\begin{align*}
  C_{\al}^{-1}\|\hat\th_{\la}\|_{H^{-1}}\leq \|A^S\hat\th_{\la}\|_{L^2}\leq C_{\al} \|\hat\th_{\la}\|_{H^{-1}}.
\end{align*}
Where $C_{\al}\geq 1$ is independent of $\nu$. 
}

Note that if $\tilde{s}=1$, both commutator $A_{\al,1}$ and $A^S$ behave similarly, which also explain the same enhanced dissipation rate for the linearized Navier-Stokes equation and its toy model.

With the above formal argument, by the definitions, formally we regard $E_0$, $E_1$, $\mathcal E_1$ and $\mathcal E_2$ as $\|\hat \theta\|^2_{L^2}$, $\|\hat \theta\|^2_{\dot H^{1}}$, $\|\hat \theta\|^2_{\dot H^{\frac{1}{4}}}$ and $\|\hat \theta\|^2_{\dot H^{-\frac{1}{2}}}$ respectively. 
% Considering the interpolation inequality for the homogeneous Sobolev space, it is reasonable to pair $\nu^{kb}$ to a $\dot H^{\frac{b}{2}}$ norm. 
To balance each part in the energy, we may formally rewrite 
\begin{align*}
	E(t)\approx \|\hat\th\|^2_{L^2}+a_1\nu^{2k}\|\hat\th\|^2_{\dot H^1}+a_2\nu^{\frac{k}{2}}\|\hat\th\|^2_{\dot H^{\frac{1}{4}}}+a_3\nu^{-k}\|\hat\th\|^2_{\dot H^{-\frac{1}{2}}},
\end{align*}
and in the sense of pairing,
\begin{align*}
  E_0\approx \nu^{2k}E_1\approx \nu^{\f{k}{2}}\mathcal{E}_1\approx \nu^{-k}\mathcal{E}_2\approx \nu^{k}\|\hat \theta\|_{\dot{H}^{\f12}}^2.
\end{align*}
The time evolution of $E_0$ and the `mixed' term $\mathcal E_1$ give us that 
\begin{align*}
\f{d}{dt}E_0=-2\nu\|\hat{\th}\|_{H^{\f12}}^2(\approx \nu^{1-k}E_0),\qquad  
\f{d}{dt}\nu^{\frac{k}{2}}\mathcal E_1=\nu^{\frac{k}{2}}\mathcal{E}_2(\approx \nu^{\frac{3k}{2}} E_0)+\text{l.o.t.},
\end{align*}
Thus we may expect that $1-k=\frac{3k}{2}$ which means $k=\frac{2}{5}$. As a result, we finally get
\begin{align*}
	\frac{d}{dt}E(t)\le-\nu^{\frac{3}{5}}E(t),
\end{align*}
which gives the enhanced dissipation rate. 

While for the transport fractional diffusion with operator $\mathcal{L}^S_{\al,\nu,\frac{1}{2}}$, a similar argument gives that
\begin{align*}
	\frac{d}{dt}E^S(t)\le-\nu^{\frac{2}{3}}E^S(t),
\end{align*}
which gives the enhanced dissipation rate. Here 
\beno
E^S(t)=\|\hat\th\|_{L^2}^2+b_1\nu^{\f23}\|\pa_y\hat\th\|_{L^2}^2+b_2\Re\langle iA^S\hat \theta, \pa_y\hat \theta\rangle+b_3\nu^{-\f23}\|A^S\hat\th\|_{L^2}^2
\eeno 
is the modified energy functional. The corresponding time weighted energy functional can be found in section 6.

\section{Energy estimates}
In this section, we study the linearized equation \eqref{eq-ft-x} and establish the energy estimates which will be used to prove Theorem \ref{Thm-cri}.

%By abuse of notation, from now on, we use $\hat \theta(\alpha,y)$ to denote $\frac{1}{2\pi\d}\int_{\mathbb{T}_{2\pi\d}} \theta(x,y)e^{-ix\al}dx$, and use $\tilde \theta(\alpha,\beta)$ to denote $\frac{1}{4\pi^2\d}\int_{\mathbb{T}_{2\pi\d}\times\bbT} \theta(x,y)e^{-i(x\al+y\beta)}dxdy$.

Recall that 
\begin{align*}
  A=\sin y \big(1-(-\Delta_{\al})^{-\frac{1}{2}}\big),\quad B=\cos y \big(1-(-\Delta_{\al})^{-\frac{1}{2}}\big),
\end{align*}
and
\begin{align*}
  E_0=\|\hat \theta\|_*^2,\quad E_1=\|\pa_y\hat \theta\|_*^2,\quad \mathcal E_1=-\frac{\al}{|\al|}\Re\langle iA\hat \theta, \pa_y\hat \theta\rangle_*,\quad \mathcal E_2=\|\hat \theta\|_*^2-\|B\hat \theta\|_*^2.
\end{align*}
We are aiming to prove the following proposition. 
\begin{proposition}\label{pro-energy}
Suppose that $|\al|\geq c_0>1$. Then it holds that 
\begin{align*}
	\frac{d}{dt}E_0=&-2\nu E_{\frac{1}{2}},\\
	\frac{d}{dt}E_1=&-2\nu E_{\frac{3}{2}}-2|\gamma|\mathcal E_1,\\
	\frac{d}{dt}\mathcal E_1\le&2\nu\|(-\Delta_\al)^{\frac{1}{4}}A\hat \theta\|_{L^2}E_{\frac{3}{2}}^{\frac{1}{2}}+c_2\nu E_{\frac{1}{2}}-|\gamma|\|A\hat \theta\|^2_{L^2}-c_3|\gamma|\sum_{\beta\in\mathbb Z}\f{2\pi\b^2}{(\al^2+\beta^2)^{\frac{1}{2}}}|\tilde\psi(\beta)|^2,\\
	\frac{d}{dt}\mathcal E_2\le&-2\nu E_0-\nu\|(-\Delta_\al)^{\frac{1}{4}}A\hat \theta\|^2_{L^2}+2c_4\nu\|A\hat \theta\|_{L^2}E_0^{\frac{1}{2}},
\end{align*}
where 
\begin{align*}
  E_{\frac{1}{2}}=2\pi\sum_{\beta\in\mathbb Z}((\al^2+\beta^2)^{\frac{1}{2}}-1)|\tilde \theta(\beta)|^2,\quad \quad
E_{\frac{3}{2}}=2\pi\sum_{\beta\in\mathbb Z}\beta^2((\al^2+\beta^2)^{\frac{1}{2}}-1)|\tilde \theta(\beta)|^2,
\end{align*}
and $c_2,c_4$ are positive constants depending only on $c_0$, and $0<c_3\le \frac{1}{2}$ is an independent constant. 
\end{proposition}

\subsection{Useful lemmas}
Let us first introduce some important lemmas.
\begin{lemma}\label{lem-equivalence}
	If $|\al|\ge c_0>1$, and $(-\Delta_{\al})^{\frac{1}{2}}\hat\psi=\hat \theta$, then
	\begin{align*}
		\|A\hat \theta\|_{L^2}^2+2\pi\sum_{\beta\in\mathbb Z}((\al^2+\beta^2)^{\frac{1}{2}}-1)|\tilde\psi(\beta)|^2\le\mathcal E_2\le 2\Big(\|A\hat \theta\|_{L^2}^2+2\pi\sum_{\beta\in\mathbb Z}((\al^2+\beta^2)^{\frac{1}{2}}-1)|\tilde\psi(\beta)|^2\Big).
	\end{align*}
\end{lemma}

\begin{proof}
From the definition, we have
\begin{align*}
  &\langle u,w\rangle_*=\big\langle \big(1-(-\Delta_{\al})^{-\frac{1}{2}}\big)u,\big(1-(-\Delta_{\al})^{-\frac{1}{2}}\big)w \big\rangle +\big\langle (-\Delta_{\al})^{-\frac{1}{2}}u,\big(1-(-\Delta_{\al})^{-\frac{1}{2}}\big)w \big\rangle,\\
  &\big\langle \big(1-(-\Delta_{\al})^{-\frac{1}{2}}\big)u,\big(1-(-\Delta_{\al})^{-\frac{1}{2}}\big)w \big\rangle=\langle Au,Aw\rangle+\langle Bu,Bw\rangle,\\
  &\langle Bu,Bw\rangle=\langle Bu,Bw\rangle_*+\langle (-\Delta_{\al})^{-\frac{1}{2}}Bu,Bw\rangle.
\end{align*}
It follows that
\begin{equation}\label{eq-uw}
  \begin{aligned}    
    &\langle u,w\rangle_*-\langle Bu,Bw\rangle_*\\
  =&\big\langle (-\Delta_{\al})^{-\frac{1}{2}}u,\big(1-(-\Delta_{\al})^{-\frac{1}{2}}\big)w \big\rangle+\langle Au,Aw\rangle+\langle (-\Delta_{\al})^{-\frac{1}{2}}Bu,Bw\rangle. 
  \end{aligned}
\end{equation}
As a result, it holds that
	\begin{align*}
		\mathcal E_2=&\|\hat \theta\|_*^2-\|B\hat \theta\|_*^2\\
    =&\langle (-\Delta_{\al})^{-\frac{1}{2}}\hat \theta, \big(1-(-\Delta_{\al})^{-\frac{1}{2}}\big)\hat \theta\rangle+\langle A\hat \theta, A\hat \theta\rangle+\langle (-\Delta_{\al})^{-\frac{1}{2}}B\hat \theta, B\hat \theta\rangle\\
		\ge&\langle A\hat \theta, A\hat \theta\rangle+\langle \hat\psi,\big((-\Delta_{\al})^{\frac{1}{2}}-1\big)\hat\psi \rangle\\
		=&\|A\hat \theta\|_{L^2}^2+2\pi\sum_{\beta\in\mathbb Z}((\al^2+\beta^2)^{\frac{1}{2}}-1)|\tilde\psi|^2.
	\end{align*}
In the last equality, we use the Plancherel theorem.

To prove the upper bound, it suffice to estimate $\langle (-\Delta_{\al})^{-\frac{1}{2}}B\hat \theta, B\hat \theta\rangle$. 

Let
\begin{align}\label{eq-def-u}
 \hat u=(1-(-\Delta_{\al})^{-\frac{1}{2}})\hat \theta,
\end{align}
then, by the fact that
  \begin{align*}
    \Big(\frac{(\al^2+\beta^2)^{\frac{1}{2}}-1}{(\al^2+(\beta-1)^2)^{\frac{1}{2}}}\Big)^2=\frac{\al^2+\b^2+1-2(\al^2+\beta^2)^{\frac{1}{2}}}{{\al^2+\b^2+1-2\b}}<1,
  \end{align*}
we deduce that
	\begin{align*}
		&\langle (-\Delta_{\al})^{-\frac{1}{2}}B\hat \theta, B\hat \theta\rangle
		=\langle (-\Delta_{\al})^{-\frac{1}{2}}\cos y \hat u, \cos y \hat u\rangle\\
    =&\int_\bbT (-\Delta_{\al})^{-\frac{1}{2}}\frac{e^{iy}+e^{-iy}}{2} \hat u\cdot\overline{\frac{e^{iy}+e^{-iy}}{2} \hat u}dy\\
		=&\frac{\pi}{2}\sum_{\beta\in\mathbb Z}\frac{|\tilde u(\beta-1)+\tilde u(\beta+1)|^2}{(\al^2+\beta^2)^{\frac{1}{2}}}\\
		\le&\pi\sum_{\beta\in\mathbb Z}\frac{|\tilde u(\beta-1)|^2+|\tilde u(\beta+1)|^2}{(\al^2+\beta^2)^{\frac{1}{2}}}\\
		=&\pi\sum_{\beta\in\mathbb Z}\frac{((\al^2+(\beta-1)^2)^{\frac{1}{2}}-1)^2|\tilde \psi(\beta-1)|^2+((\al^2+(\beta+1)^2)^{\frac{1}{2}}-1)^2|\tilde \psi(\beta+1)|^2}{(\al^2+\beta^2)^{\frac{1}{2}}}\\	
		=&2\pi\sum_{\beta\in\mathbb Z}\Big(\frac{(\al^2+\beta^2)^{\frac{1}{2}}-1}{2(\al^2+(\beta+1)^2)^{\frac{1}{2}}}+\frac{(\al^2+\beta^2)^{\frac{1}{2}}-1}{2(\al^2+(\beta-1)^2)^{\frac{1}{2}}}\Big)((\al^2+\beta^2)^{\frac{1}{2}}-1)|\tilde \psi(\beta)|^2\\
		\le&2\pi\sum_{\beta\in\mathbb Z}((\al^2+\beta^2)^{\frac{1}{2}}-1)|\tilde\psi|^2.
	\end{align*}
  This completes the proof of this lemma.
\end{proof}
\begin{lemma}\label{lem-est-me2}
	Under the same assumptions of Lemma \ref{lem-equivalence}, it holds that
	\begin{align*}
		\mathcal E_2\le& c_1\Big(|\al|^{\frac{1}{2}}\|A\hat \theta\|_{L^2}^2+\sum_{\beta\in\mathbb Z}\frac{2\pi|\al|^{\frac{1}{2}}\beta^2}{(\al^2+\beta^2)^{\frac{1}{2}}}|\tilde\psi(\beta)|^2\Big),
	\end{align*}
	where $c_1$ is a constant depending only on $c_0$.
\end{lemma}
\begin{proof}
As Lemma \ref{lem-equivalence} shows that
	\begin{align*}
		\mathcal E_2\le 2\Big(\|A\hat \theta\|_{L^2}^2+2\pi\sum_{\beta\in\mathbb Z}((\al^2+\beta^2)^{\frac{1}{2}}-1)|\tilde\psi|^2(\al,\beta)\Big),
	\end{align*}
it suffice to prove that
	\begin{align*}
		2\pi\sum_{\beta\in\mathbb Z}((\al^2+\beta^2)^{\frac{1}{2}}-1)|\tilde\psi(\beta)|^2\le C\Big(|\al|^{\frac{1}{2}}\|A\hat \theta\|_{L^2}^2+\sum_{\beta\in\mathbb Z}\frac{2\pi|\al|^{\frac{1}{2}}\beta^2}{(\al^2+\beta^2)^{\frac{1}{2}}}|\tilde\psi(\beta)|^2\Big)
	\end{align*}
for some constant $C$ depends only on $c_0$.

For $|\beta|\ge|\al|$, it is clear 
	\begin{align*}
		((\al^2+\beta^2)^{\frac{1}{2}}-1)|\tilde\psi(\beta)|^2\le2\frac{|\al|^{\frac{1}{2}}\beta^2}{(\al^2+\beta^2)^{\frac{1}{2}}}|\tilde\psi(\beta)|^2,
	\end{align*}
so we only need to focus on the case $|\beta|\le|\al|$.

However, when $|\beta|\le|\al|$, it holds that
  \begin{align*}
    \left((\al^2+\beta^2)^{\frac{1}{2}}-1-\frac{|\al|^{\frac{1}{2}}\beta^2}{(\al^2+\beta^2)^{\frac{1}{2}}}\right)|\tilde\psi(\beta)|^2
    =\left(\frac{\al^2+\b^2-|\al|^{\frac{1}{2}}\beta^2}{(\al^2+\beta^2)^{\frac{1}{2}}}-1\right)|\tilde\psi(\beta)|^2
    \le(|\al|-1)|\tilde\psi(\beta)|^2,
  \end{align*}
so the problem reduces to the estimate of
  \begin{align*}
    \sum_{\beta\in\mathbb Z,|\beta|\le|\al|}(|\al|-1)|\tilde\psi(\beta)|^2.
  \end{align*}
We start from the case $|\al|\ge100$, for which $|\al|-1\ge9|\al|^{\frac{1}{2}}$. We define
  \begin{equation*}
    f(\beta)=\left\{
      \begin{array}{ll}
        1,&|\beta|\le|\al|,\\
        2- \frac{|\beta|}{|\al|},&  |\al|<|\beta|\le2|\al|,\\
        0,&2|\al|<|\beta|,
      \end{array}
    \right.
  \end{equation*}
and deduce that
\begin{equation}\label{eq-estimate-psi}
    \begin{aligned}
    &\sum_{\beta\in\mathbb Z,|\beta|\le|\al|}(|\al|-1)|\tilde\psi(\beta)|^2\\
    \le&\sum_{\beta\in\mathbb Z}(|\al|-1)|f(\beta)\tilde\psi(\beta)|^2\\
    =&\frac{1}{2}\sum_{\beta\in\mathbb Z}(|\al|-1)\beta|f(\beta-1)\tilde\psi(\beta-1)|^2-(|\al|-1)\beta|f(\beta+1)\tilde\psi(\beta+1)|^2\\
    =&\frac{1}{2}\sum_{\beta\in\mathbb Z}(|\al|-1)\beta\big(f(\beta-1)\tilde\psi(\beta-1)-f(\beta+1)\tilde\psi(\beta+1)\big)\\
    &\qquad\qquad\qquad\qquad\qquad\qquad\qquad\cdot\overline{\big(f(\beta-1)\tilde\psi(\beta-1)+f(\beta+1)\tilde\psi(\beta+1)\big)}\\
    \le&\frac{1}{4}|\al|^{\frac{1}{2}}(|\al|-1)^2\sum_{\beta\in\mathbb Z}|f(\beta-1)\tilde\psi(\beta-1)-f(\beta+1)\tilde\psi(\beta+1)|^2\\
    &+\frac{1}{4}|\al|^{-\frac{1}{2}}\sum_{\beta\in\mathbb Z}|\beta|^2|f(\beta-1)\tilde\psi(\beta-1)+f(\beta+1)\tilde\psi(\beta+1)|^2\\
    \eqdefa&I_1+I_2.
  \end{aligned}
\end{equation}

From the definition of $f$, one can see that
  \begin{align*}
    I_1=&\frac{1}{4}|\al|^{\frac{1}{2}}(|\al|-1)^2\sum_{\beta\in\mathbb Z}|f(\beta-1)\tilde\psi(\beta-1)-f(\beta+1)\tilde\psi(\beta+1)|^2\\
    =&\frac{1}{4}|\al|^{\frac{1}{2}}(|\al|-1)^2\sum_{\beta\in\mathbb Z}\Big|f(\beta)\tilde\psi(\beta-1)-f(\beta)\tilde\psi(\beta+1)\\
    &\qquad\qquad\qquad\qquad\qquad +\big(f(\beta-1)-f(\beta)\big)\tilde\psi(\beta-1)-\big(f(\beta+1)-f(\beta)\big)\tilde\psi(\beta+1)\Big|^2\\
    \le&\frac{1}{2}|\al|^{\frac{1}{2}}(|\al|-1)^2\sum_{\beta\in\mathbb Z}f^2(\beta)|\tilde\psi(\beta-1)-\tilde\psi(\beta+1)|^2+2 \frac{(|\al|-1)^2}{|\al|^{3/2}}  \sum_{\beta,|\al|\le|\beta|\le2|\al|}|\tilde\psi(\beta)|^2\\
    \le&\frac{2}{\pi}|\al|^{\frac{1}{2}}\|A\hat \theta\|^2_{L^2}+4|\al|^{\frac{1}{2}}\sum_{|\beta|\le|\al|}|\tilde\psi(\beta)|^2+10\sum_{\beta\in\mathbb Z}\frac{|\al|^{\frac{1}{2}}\beta^2}{(\al^2+\beta^2)^{\frac{1}{2}}}|\tilde\psi(\beta)|^2.
  \end{align*}
In the last inequality, we use the fact that
  \begin{equation}\label{ineq-psi-A}
    \ |\al|^{\frac{1}{2}}\sum_{\beta\in\mathbb Z}\big((\al^2+\beta^2)^{\frac{1}{2}}-1\big)^2|\tilde\psi(\beta-1)-\tilde\psi(\beta+1)|^2\le \frac{4}{\pi}|\al|^{\frac{1}{2}}\|A\hat \theta\|^2_{L^2}+8|\al|^{\frac{1}{2}}\sum_{\beta\in\mathbb Z}|\tilde\psi(\beta)|^2.
  \end{equation}
Indeed, from the definition of $A$, we have
	\begin{align*}
		&|\al|^{\frac{1}{2}}\|A\hat \theta\|^2_{L^2}		
		=|\al|^{\frac{1}{2}}\|\sin y\big((-\Delta_\alpha)^{\frac{1}{2}}-1\big)\hat\psi\|^2_{L^2}\\
    =&|\al|^{\frac{1}{2}}\int_\bbT \frac{e^{iy}-e^{-iy}}{2i}\big((-\Delta_\alpha)^{\frac{1}{2}}-1\big)\hat\psi\cdot\overline{\frac{e^{iy}-e^{-iy}}{2i}\big((-\Delta_\alpha)^{\frac{1}{2}}-1\big)\hat\psi}dy\\
		=&\frac{\pi|\al|^{\frac{1}{2}}}{2}\sum_{\beta\in\mathbb Z}\bigg|\big((\al^2+\beta^2)^{\frac{1}{2}}-1\big)\big(\tilde\psi(\beta-1)-\tilde\psi(\beta+1)\big)\\
		&\quad\quad\qquad+\frac{(1-2\beta)\tilde\psi(\beta-1)}{(\al^2+\beta^2)^{\frac{1}{2}}+(\al^2+(\beta-1)^2)^{\frac{1}{2}}}+\frac{(1+2\beta)\tilde\psi(\beta+1)}{(\al^2+\beta^2)^{\frac{1}{2}}+(\al^2+(\beta+1)^2)^{\frac{1}{2}}}\bigg|^2\\
    \ge&\frac{\pi|\al|^{\frac{1}{2}}}{2}\sum_{\beta\in\mathbb Z}\Big(\frac{1}{2}\Big|\big((\al^2+\beta^2)^{\frac{1}{2}}-1\big)\big(\tilde\psi(\beta-1)-\tilde\psi(\beta+1)\big)\Big|^2-\big(|\tilde\psi(\beta-1)|+|\tilde\psi(\beta+1)|\big)^2\Big)\\
		\ge&\frac{\pi|\al|^{\frac{1}{2}}}{4}\sum_{\beta\in\mathbb Z}\big((\al^2+\beta^2)^{\frac{1}{2}}-1\big)^2|\tilde\psi(\beta-1)-\tilde\psi(\beta+1)|^2-2\pi|\al|^{\frac{1}{2}}\sum_{\beta\in\mathbb Z}|\tilde\psi(\beta)|^2.
	\end{align*}

Now we turn to the estimate of $I_2$ in \eqref{eq-estimate-psi}. It holds that
	\begin{align*}
		I_2=&\frac{1}{4}|\al|^{-\frac{1}{2}}\sum_{\beta\in\mathbb Z}|\beta|^2|f(\beta-1)\tilde\psi(\beta-1)+f(\beta+1)\tilde\psi(\beta+1)|^2\\
		=&\frac{1}{4}|\al|^{-\frac{1}{2}}\sum_{\beta\in\mathbb Z}\Big|\Big(f(\beta-1)\tilde\psi(\beta-1)-f(\beta+1)\tilde\psi(\beta+1)\Big)\\
		&\qquad\qquad\qquad\qquad+\Big((\beta-1)f(\beta-1)\tilde\psi(\beta-1)+(\beta+1)f(\beta+1)\tilde\psi(\beta+1)\Big)\Big|^2\\
		\le&\frac{1}{2}|\al|^{-\frac{1}{2}}\sum_{\beta\in\mathbb Z}f^2(\beta)|\tilde\psi(\beta-1)-\tilde\psi(\beta+1)|^2+2|\al|^{-\frac{1}{2}}\sum_{\beta\in\mathbb Z}|\beta|^2f^2(\beta)|\tilde\psi(\beta)|^2\\
		\le&\frac{1}{4\pi}|\al|^{\frac{1}{2}}\|A\hat \theta\|^2_{L^2}+\frac{1}{2}|\al|^{\frac{1}{2}}\sum_{\beta\in\mathbb Z}|\tilde\psi(\beta)|^2+5\sum_{\beta\in\mathbb Z}\frac{|\al|^{\frac{1}{2}}\beta^2}{(\al^2+\beta^2)^{\frac{1}{2}}}|\tilde\psi(\beta)|^2\\
    \le&\frac{1}{4\pi}|\al|^{\frac{1}{2}}\|A\hat \theta\|^2_{L^2}+\frac{1}{2}|\al|^{\frac{1}{2}}\sum_{\beta\in\mathbb Z,|\beta|\le|\al|}|\tilde\psi(\beta)|^2+6\sum_{\beta\in\mathbb Z}\frac{|\al|^{\frac{1}{2}}\beta^2}{(\al^2+\beta^2)^{\frac{1}{2}}}|\tilde\psi(\beta)|^2.
	\end{align*}
In the second last inequality, we used \eqref{ineq-psi-A} and the fact that $f(\beta)=0$ for $|\beta|\ge2|\al|$.

From the above estimates, we have
	\begin{align*}
		&\sum_{\beta\in\mathbb Z,|\beta|\le|\al|}(|\al|-1)|\tilde\psi(\beta)|^2\\
		\le&\frac{9}{4\pi}|\al|^{\frac{1}{2}}\|A\hat \theta\|^2_{L^2}+\frac{9}{2}|\al|^{\frac{1}{2}}\sum_{\beta\in\mathbb Z,|\beta|\le|\al|}|\tilde\psi(\beta)|^2+16\sum_{\beta\in\mathbb Z}\frac{|\al|^{\frac{1}{2}}\beta^2}{(\al^2+\beta^2)^{\frac{1}{2}}}|\tilde\psi(\beta)|^2.
	\end{align*}
Then, it follows from the fact $|\al|\ge100$ and $|\al|-1\ge9|\al|^{\frac{1}{2}}$ that
	\begin{align*}
		\sum_{\beta\in\mathbb Z,|\beta|\le|\al|}(|\al|-1)|\tilde\psi(\beta)|^2\le \frac{9}{2\pi}|\al|^{\frac{1}{2}}\|A\hat \theta\|^2_{L^2}+32\sum_{\beta\in\mathbb Z}\frac{|\al|^{\frac{1}{2}}\beta^2}{(\al^2+\beta^2)^{\frac{1}{2}}}|\tilde\psi(\beta)|^2.
	\end{align*}
	For $|\al|<100$ and $0<|\beta|<|\al|$, it is obvious that
	\begin{align*}
		(|\al|-1)|\tilde\psi(\beta)|^2\le2000\frac{|\al|^{\frac{1}{2}}\beta^2}{(\al^2+\beta^2)^{\frac{1}{2}}}|\tilde\psi(\beta)|^2.
	\end{align*}
The rest is the case $1<c_0\le|\al|<100$ and $|\beta|=0$. Similar to the proof of \eqref{ineq-psi-A}, we have
\begin{align*}
		\|A\hat \theta\|^2_{L^2}		
		=&\|\sin y\big((-\Delta_\alpha)^{\frac{1}{2}}-1\big)\psi\|^2_{L^2}\\
		\ge&2\pi\big|(|\al|-1)\tilde\psi(0)-\big((\al^2+4)^{\frac{1}{2}}-1\big)\tilde\psi(2)\big|^2\\
		\ge&\pi(|\al|-1)^2|\tilde\psi(0)|^2-2\pi\big((\al^2+4)^{\frac{1}{2}}-1\big)^2|\tilde\psi(2)|^2.
	\end{align*}
It follows that
	\begin{align*}
		(|\al|-1)|\tilde\psi(0)|^2\le \frac{1}{\pi(\al-1)}\|A\hat \theta\|^2_{L^2}+\frac{2000}{\al-1}\frac{4|\al|^{\frac{1}{2}}}{(\al^2+4)^{\frac{1}{2}}}|\tilde\psi(2)|^2.
	\end{align*}
Combining these results, we arrive at the conclusion of this lemma.
\end{proof}
\begin{lemma}\label{lem-com-ba}
Under the same assumptions of Lemma \ref{lem-equivalence}, it holds that
  \begin{align*}
  \Re\langle (BA-AB)\hat \theta, \pa_y\hat \theta\rangle_*\le-\frac{1}{C}\sum_{\beta\in\mathbb Z}\frac{2\pi\beta^2}{(\al^2+\beta^2)^{\frac{1}{2}}}|\tilde\psi(\beta)|^2,
\end{align*}
where $C>0$ is a independent constant.
\end{lemma}
\begin{proof}
By the definitions, it is clear that
  \begin{align*}
  &BA-AB\\
  =&\cos y\big(1-(-\Delta_{\al})^{-\frac{1}{2}}\big)\sin y\big(1-(-\Delta_{\al})^{-\frac{1}{2}}\big)-\sin y\big(1-(-\Delta_{\al})^{-\frac{1}{2}}\big)\cos y\big(1-(-\Delta_{\al})^{-\frac{1}{2}}\big)\\
  =&\Big(-\cos y(-\Delta_{\al})^{-\frac{1}{2}}\sin y+\sin y(-\Delta_{\al})^{-\frac{1}{2}}\cos y\Big)\big(1-(-\Delta_{\al})^{-\frac{1}{2}}\big).
\end{align*}
A direct calculation shows that
\begin{align*}
  &\cos y(-\Delta_{\al})^{-\frac{1}{2}}\sin y-\sin y(-\Delta_{\al})^{-\frac{1}{2}}\cos y\\
  =&\Big((e^{iy}+e^{-iy})(-\Delta_\al)^{-\frac{1}{2}}(e^{iy}-e^{-iy})-(e^{iy}-e^{-iy})(-\Delta_\al)^{-\frac{1}{2}}(e^{iy}+e^{-iy})\Big)/(4i)\\
  =&\Big(e^{-iy}(-\Delta_\al)^{-\frac{1}{2}}e^{iy}-e^{iy}(-\Delta_\al)^{-\frac{1}{2}}e^{-iy}\Big)/(2i)\\
  =&\Big(\big(e^{-iy}(-\Delta_\al)^{\frac{1}{2}}e^{iy}\big)^{-1}-\big(e^{iy}(-\Delta_\al)^{\frac{1}{2}}e^{-iy}\big)^{-1}\Big)/(2i)\\
  =&\Big((-\Delta_{\al,1})^{-\frac{1}{2}}-(-\Delta_{\al,-1})^{-\frac{1}{2}}\Big)/(2i)\\
  =&\Big((-\Delta_{\al,-1})^{\frac{1}{2}}-(-\Delta_{\al,1})^{\frac{1}{2}}\Big)(-\Delta_{\al,1})^{-\frac{1}{2}}(-\Delta_{\al,-1})^{-\frac{1}{2}}/(2i),
\end{align*}
where
\begin{align*}
  (-\Delta_{\al,1})^{\frac{1}{2}}=e^{-iy}(-\Delta_\al)^{\frac{1}{2}}e^{iy},\quad (-\Delta_{\al,-1})^{\frac{1}{2}}=e^{iy}(-\Delta_\al)^{\frac{1}{2}}e^{-iy}.
\end{align*}
From Fourier point of view, one can see that
\begin{align*}
  \mathcal F_y\big((-\Delta_{\al,1})^{\frac{1}{2}}\hat f\big)(\beta)&=(\al^2+(\beta+1)^2)^{\frac{1}{2}}\tilde f(\beta).
\end{align*}
It follows that
\begin{align*}
  &\mathcal F_y\Big(\cos y(-\Delta_{\al})^{-\frac{1}{2}}\sin y-\sin y(-\Delta_{\al})^{-\frac{1}{2}}\cos y\Big)\\
  =&\Big((\al^2+(\beta-1)^2)^{\frac{1}{2}}-(\al^2+(\beta+1)^2)^{\frac{1}{2}}\Big)(\al^2+(\beta-1)^2)^{-\frac{1}{2}}(\al^2+(\beta+1)^2)^{-\frac{1}{2}}/(2i)\\
  =&2i\beta\Big((\al^2+(\beta-1)^2)^{\frac{1}{2}}+(\al^2+(\beta+1)^2)^{\frac{1}{2}}\Big)^{-1}(\al^2+(\beta-1)^2)^{-\frac{1}{2}}(\al^2+(\beta+1)^2)^{-\frac{1}{2}}.
\end{align*}
Thus, by using the fact that $|\al|>1$, one can deduce that 
\begin{align*}
  &\Re\langle (BA-AB)\hat \theta, \pa_y\hat \theta\rangle_*\\
  =&\Re\Big\langle \Big(-\cos y(-\Delta_{\al})^{-\frac{1}{2}}\sin y+\sin y(-\Delta_{\al})^{-\frac{1}{2}}\cos y\Big)\big(1-(-\Delta_{\al})^{-\frac{1}{2}}\big)\hat \theta, \pa_y\hat \theta\Big\rangle_*\\
  =&\Re\left\langle -\big((-\Delta_{\al,-1})^{\frac{1}{2}}-(-\Delta_{\al,1})^{\frac{1}{2}}\big)(-\Delta_{\al,1})^{-\frac{1}{2}}(-\Delta_{\al,-1})^{-\frac{1}{2}}\big(1-(-\Delta_{\al})^{-\frac{1}{2}}\big)\frac{\hat \theta}{2i}, \big(1-(-\Delta_{\al})^{-\frac{1}{2}}\big)\pa_y\hat \theta\right\rangle\\
  =&-4\pi\sum_{\beta\in\mathbb Z}\frac{(\al^2+(\beta-1)^2)^{-\frac{1}{2}}(\al^2+(\beta+1)^2)^{-\frac{1}{2}}}{(\al^2+(\beta-1)^2)^{\frac{1}{2}}+(\al^2+(\beta+1)^2)^{\frac{1}{2}}}\big(1-(\al^2+\beta^2)^{-\frac{1}{2}}\big)^2\beta^2|\tilde \theta(\beta)|^2\\
  =&-4\pi\sum_{\beta\in\mathbb Z}\frac{((\al^2+\beta^2)^{\frac{1}{2}}-1)^2}{(\al^2+(\beta-1)^2)^{\frac{1}{2}}+(\al^2+(\beta+1)^2)^{\frac{1}{2}}}\frac{\beta^2|\tilde\psi(\beta)|^2}{(\al^2+(\beta-1)^2)^{\frac{1}{2}}(\al^2+(\beta+1)^2)^{\frac{1}{2}}}\\
  \le&-\frac{1}{C}\sum_{\beta\in\mathbb Z}\frac{2\pi\beta^2}{(\al^2+\beta^2)^{\frac{1}{2}}}|\tilde\psi(\beta)|^2.
\end{align*}
This completes the proof.
\end{proof}

\subsection{Energy estimate} Recall the energy functional 
\begin{align*}
  \Phi(t)=E_0+a_1\nu^{2}t^2E_1+a_2\nu^{2}t^3\mathcal E_1+a_3\nu^{2}t^4\mathcal E_2,
\end{align*}
where the coefficients $a_1,a_2,a_3$ are determined in the proof. It is natural to study the time evolution of the energy functional $\Phi(t)$:
\begin{align*}
\f{d}{dt}\Phi(t)=\f{d}{dt}E_0&+a_1\nu^2t^2\f{d}{dt}E_1+a_2\nu^2t^3\f{d}{dt}\mathcal{E}_1+a_3\nu^2t^4\f{d}{dt}\mathcal{E}_2\\
&+2a_1\nu^2tE_1+3a_2\nu^2t^2\mathcal{E}_1+4a_3\nu^2t^3\mathcal{E}_2. 
\end{align*}

Now, we are in a position to prove Proposition \ref{pro-energy}, and to estimate the time evolution of $E_0$, $E_1$, $\mathcal{E}_1$ and $\mathcal{E}_2$. 
\begin{proof}[Proof of Proposition \ref{pro-energy}]
We only prove for the case $\al>1$, for which $\mathcal E_1=-\Re\langle iA\hat \theta, \pa_y\hat \theta\rangle_*$ and $\gamma>0$. The proof for $\al<-1$ is the same. 

Recalling that $\mathcal L_\nu(\al,t)=\nu(-\Delta_{\al})^{\frac{1}{2}}+i\gamma(t)B$ and $E_0=\|\hat \theta\|^2_*=2\pi\sum\limits_{\beta\in\mathbb Z}\frac{(\al^2+\beta^2)^{\frac{1}{2}}-1}{(\al^2+\beta^2)^{\frac{1}{2}}}|\tilde \theta(\beta)|^2$, we have
\begin{align*}
	\frac{d}{dt}E_0=&2\Re\langle \pa_t\hat \theta, \hat \theta\rangle_*=-2\Re\langle \mathcal L_\nu\hat \theta, \hat \theta\rangle_*\\
	=&-2\nu\Re\langle (-\Delta_\alpha)^{\frac{1}{2}}\hat \theta, \hat \theta\rangle_*-2\gamma\Re\langle iB\hat \theta, \hat \theta\rangle_*\\
	=&-2\nu\langle  (-\Delta_\alpha)^{\frac{1}{4}}\hat \theta,  (-\Delta_\alpha)^{\frac{1}{4}}\hat \theta\rangle_*\\
  =&-4\pi\nu\sum_{\beta\in\mathbb Z}((\al^2+\beta^2)^{\frac{1}{2}}-1)|\tilde \theta(\beta)|^2=-2\nu E_{\frac{1}{2}}.
\end{align*}
Here we use the fact that $B$ is symmetric with respect to $\langle\cdot,\cdot\rangle_*$.

For $E_1=\|\pa_y\hat \theta\|^2_*=2\pi\sum\limits_{\beta\in\mathbb Z}\frac{\beta^2((\al^2+\beta^2)^{\frac{1}{2}}-1)}{(\al^2+\beta^2)^{\frac{1}{2}}}|\tilde\theta(\beta)|^2$, one can deduce that
\begin{align*}
	\frac{d}{dt}E_1=&2\Re\langle \pa_y\pa_t\hat \theta, \pa_y\hat \theta\rangle_*=-2\Re\langle \pa_y\mathcal L_\nu\hat \theta, \hat \theta\rangle_*\\
	=&-2\nu\Re\langle (-\Delta_\alpha)^{\frac{1}{2}}\pa_y\hat \theta, \pa_y\hat \theta\rangle_*-2\gamma\Re\langle i\pa_y(B\hat \theta), \pa_y\hat \theta\rangle_*\\
	=&-2\nu\Re\langle (-\Delta_\alpha)^{\frac{1}{2}}\pa_y\hat \theta, \pa_y\hat \theta\rangle_*-2\gamma\Re\langle iB\pa_y\hat \theta, \pa_y\hat \theta\rangle_*+2\gamma\Re\langle iA\hat \theta, \pa_y\hat \theta\rangle_*\\
	=&-2\nu\Re\langle (-\Delta_\alpha)^{\frac{1}{2}}\pa_y\hat \theta, \pa_y\hat \theta\rangle_*+2\gamma\Re\langle iA\hat \theta, \pa_y\hat \theta\rangle_*\\
  =&-4\pi\nu\sum_{\beta\in\mathbb Z}\beta^2((\al^2+\beta^2)^{\frac{1}{2}}-1)|\tilde \theta(\beta)|^2+2\gamma\Re\langle iA\hat \theta, \pa_y\hat \theta\rangle_*\\
	=&-2\nu E_{\frac{3}{2}}-2\gamma\mathcal E_1.
\end{align*}

Next, we turn to $\mathcal E_1$ and deduce that
\begin{equation}\label{eq-mE1}
  \begin{aligned}    
      \frac{d}{dt}\mathcal E_1=&-\Re\langle iA\pa_t\hat \theta, \pa_y\hat \theta\rangle_*-\Re\langle iA\hat \theta, \pa_y\pa_t\hat \theta\rangle_*\\
  =&\nu\Re\langle iA(-\Delta_\alpha)^{\frac{1}{2}}\hat \theta, \pa_y\hat \theta\rangle_*-\gamma\Re\langle AB\hat \theta, \pa_y\hat \theta\rangle_*\\
  &+\nu\Re\langle iA\hat \theta, (-\Delta_\alpha)^{\frac{1}{2}}\pa_y\hat \theta\rangle_*+\gamma\Re\langle A\hat \theta, B\pa_y\hat \theta\rangle_*-\gamma\Re\langle A\hat \theta, A\hat \theta\rangle_*\\
  =&\nu\Re\langle iA(-\Delta_\alpha)^{\frac{1}{2}}\hat \theta, \pa_y\hat \theta\rangle_*+\nu\Re\langle iA\hat \theta, (-\Delta_\alpha)^{\frac{1}{2}}\pa_y\hat \theta\rangle_*\\
  &+\gamma\Re\langle (BA-AB)\hat \theta, \pa_y\hat \theta\rangle_*-\gamma\Re\langle A\hat \theta, A\hat \theta\rangle_*.
  \end{aligned}
\end{equation}
For the first two terms on the right hand side, we have
\begin{align*}
	&\Re\langle iA(-\Delta_\alpha)^{\frac{1}{2}}\hat \theta, \pa_y\hat \theta\rangle_*+\Re\langle iA\hat \theta, (-\Delta_\alpha)^{\frac{1}{2}}\pa_y\hat \theta\rangle_*\\
	=&2\Re\langle i (-\Delta_\alpha)^{\frac{1}{4}}A\hat \theta,  (-\Delta_\alpha)^{\frac{1}{4}}\pa_y\hat \theta\rangle_*+\Re\left\langle i\Big(A(-\Delta_\alpha)^{\frac{1}{2}}-(-\Delta_\alpha)^{\frac{1}{2}}A\Big)\hat \theta, \pa_y\hat \theta\right\rangle_*\\
  \le&2\|(-\Delta_\alpha)^{\frac{1}{4}}A\hat \theta\|_{L^2}E_{\frac{3}{2}}^{\frac{1}{2}}+\Re\left\langle i\Big(A(-\Delta_\alpha)^{\frac{1}{2}}-(-\Delta_\alpha)^{\frac{1}{2}}A\Big)\hat \theta, \pa_y\hat \theta\right\rangle_*.
\end{align*}
Due to the fine property of the communicator, the last term is a lower order term. Indeed, recalling that $\hat{u}=(1-(-\Delta_{\al})^{-\frac{1}{2}})\hat \theta$, one can deduce that
\begin{align*}
	&\left|\Re\left\langle i\Big(A(-\Delta_\alpha)^{\frac{1}{2}}-(-\Delta_\alpha)^{\frac{1}{2}}A\Big)\hat \theta, \pa_y\hat \theta\right\rangle_*\right|\\
	=&\left|\Re\left\langle i\Big(\sin y(-\Delta_\alpha)^{\frac{1}{2}}-(-\Delta_\alpha)^{\frac{1}{2}}\sin y\Big)\hat u, \pa_y\hat u\right\rangle\right|\\
	=&\frac{1}{2}\Big|\int_\bbT i\Big(\sin y(-\Delta_\alpha)^{\frac{1}{2}}-(-\Delta_\alpha)^{\frac{1}{2}}\sin y\Big)\hat u\cdot\overline{\pa_y\hat u}dy\\
	&\qquad\qquad\qquad\qquad\qquad+\int_\bbT \pa_y\hat u\cdot\overline{i\Big(\sin y(-\Delta_\alpha)^{\frac{1}{2}}-(-\Delta_\alpha)^{\frac{1}{2}}\sin y\Big)\hat u}dy\Big|\\
	=&\frac{\pi}{2}\Big|\sum_{\beta\in\mathbb Z}i\beta\Big((\al^2+(\beta-1)^2)^{\frac{1}{2}}-(\al^2+\beta^2)^{\frac{1}{2}}\Big)\big(\tilde u(\beta)\overline{\tilde u(\beta-1)}-\tilde u(\beta-1)\overline{\tilde u(\beta)}\big)\\
	&\qquad\qquad\qquad-\sum_{\beta\in\mathbb Z}i\beta\Big((\al^2+(\beta+1)^2)^{\frac{1}{2}}-(\al^2+\beta^2)^{\frac{1}{2}}\Big)\big(\tilde u(\beta)\overline{\tilde u(\beta+1)}-\tilde u(\beta+1)\overline{\tilde u(\beta)}\big)\Big|\\
	=&\frac{\pi}{2}\Big|\sum_{\beta\in\mathbb Z}i\beta \frac{1-2\beta}{(\al^2+(\beta-1)^2)^{\frac{1}{2}}+(\al^2+\beta^2)^{\frac{1}{2}}}\big(\tilde u(\beta)\overline{\tilde u(\beta-1)}-\tilde u(\beta-1)\overline{\tilde u(\beta)}\big)\\
	&\qquad\qquad\qquad-\sum_{\beta\in\mathbb Z}i\beta\frac{1+2\beta}{(\al^2+(\beta+1)^2)^{\frac{1}{2}}+(\al^2+\beta^2)^{\frac{1}{2}}}\big(\tilde u(\beta)\overline{\tilde u(\beta+1)}-\tilde u(\beta+1)\overline{\tilde u(\beta)}\big)\Big|\\
	\le&c_2E_{\frac{1}{2}},
\end{align*}
where $c_2$ is a constant depends only on $c_0$.

By using Lemma \ref{lem-com-ba} to estimate the third term on the right hand side of \eqref{eq-mE1}, we conclude that
\begin{align*}
	\frac{d}{dt}\mathcal E_1\le&2\nu\|(\al^2-\pa_y^2)^{1/4}A\hat \theta\|_{L^2}E_{\frac{3}{2}}^{\frac{1}{2}}+c_2\nu E_{\frac{1}{2}}-\gamma\|A\hat \theta\|^2_{L^2}-c_3\gamma\sum_{\beta\in\mathbb Z}\frac{2\pi\beta^2}{(\al^2+\beta^2)^{\frac{1}{2}}}|\tilde\psi|^2.
\end{align*}
Here we emphasize that $0<c_3\le \frac{1}{2}$ is an independent constant.

For $\mathcal E_2=\|\hat \theta\|_*^2-\|B\hat \theta\|_*^2$, by using \eqref{eq-uw} and the fact that $B$ is symmetric, we have
\begin{align*}
	\frac{d}{dt}\mathcal E_2=&2\Re\langle \pa_t\hat \theta, \hat \theta\rangle_*-2\Re\langle B\pa_t\hat \theta, B\hat \theta\rangle_*\\
	=&-2\nu\Re\langle (-\Delta_\alpha)^{\frac{1}{2}}\hat \theta, \hat \theta\rangle_*+2\nu\Re\langle B(-\Delta_\alpha)^{\frac{1}{2}}\hat \theta, B\hat \theta\rangle_*-2\gamma\Re\langle iB\hat \theta, \hat \theta\rangle_*+2\gamma\Re\langle iBB\hat \theta, B\hat \theta\rangle_*\\
	=&-2\nu\Re\langle (-\Delta_\alpha)^{\frac{1}{2}}\hat \theta, \hat \theta\rangle_*+2\nu\Re\langle B(-\Delta_\alpha)^{\frac{1}{2}}\hat \theta, B\hat \theta\rangle_*\\
	=&-2\nu\Big(\Re\langle (-\Delta_{\al})^{-\frac{1}{2}}B(-\Delta_\alpha)^{\frac{1}{2}}\hat \theta, B\hat \theta\rangle+\Re\langle A(-\Delta_\alpha)^{\frac{1}{2}}\hat \theta, A\hat \theta\rangle+\Re\langle \hat \theta, \hat \theta-(-\Delta_{\al})^{-\frac{1}{2}}\hat \theta\rangle\Big).
\end{align*}
A direct calculation shows that
\begin{align*}
	&\Re\langle A(-\Delta_\alpha)^{\frac{1}{4}}\hat \theta, A\hat \theta\rangle\\
	=&\langle (-\Delta_\alpha)^{\frac{1}{4}}A\hat \theta, (-\Delta_\alpha)^{\frac{1}{4}}A\hat \theta\rangle+\Re\langle \big(A(-\Delta_\alpha)^{\frac{1}{2}}-(-\Delta_\alpha)^{\frac{1}{2}}A\big)\hat \theta, A\hat \theta\rangle\\
  \ge&\langle (-\Delta_\alpha)^{\frac{1}{4}}A\hat \theta, (-\Delta_\alpha)^{\frac{1}{4}}A\hat \theta\rangle-\big\|\big(A(-\Delta_\alpha)^{\frac{1}{2}}-(-\Delta_\alpha)^{\frac{1}{2}}A\big)\hat \theta\big\|_{L^2}\|A\hat \theta\|_{L^2}.
\end{align*}
Recalling that $\hat u=(1-(-\Delta_{\al})^{-\frac{1}{2}})\hat \theta$, one can see that
\begin{align*}
	&\left\|\Big(A(-\Delta_\alpha)^{\frac{1}{2}}-(-\Delta_\alpha)^{\frac{1}{2}}A\Big)\hat \theta\right\|^2_{L^2}\\
	=&\left\|\Big(\sin y(-\Delta_\alpha)^{\frac{1}{2}}-(-\Delta_\alpha)^{\frac{1}{2}}\sin y\Big)\hat u\right\|^2_{L^2}\\
	=&\frac{\pi}{2}\sum_{\beta\in\mathbb Z}\left|\frac{(1-2\beta)\tilde u(\beta-1)}{(\al^2+(\beta-1)^2)^{\frac{1}{2}}+(\al^2+\beta^2)^{\frac{1}{2}}}-\frac{(1+2\beta)\tilde u(\beta+1)}{(\al^2+(\beta+1)^2)^{\frac{1}{2}}+(\al^2+\beta^2)^{\frac{1}{2}}}\right|^2\\
	\le&2\pi\sum_{\beta\in\mathbb Z}|\tilde u(\beta)|^2=2\pi\sum_{\beta\in\mathbb Z}\frac{\big((\al^2+\beta^2)^{\frac{1}{2}}-1\big)^2}{\al^2+\beta^2}|\tilde \theta(\beta)|^2\\
	\le& c_4^2E_0,
\end{align*}
where $c_4$ is a constant depending only on $c_0$.

Similarly, we deduce that
\begin{align*}
	&\Re\langle (-\Delta_{\al})^{-\frac{1}{2}}B(-\Delta_\alpha)^{\frac{1}{2}}\hat \theta, B\hat \theta\rangle\\
	=&\Re\langle B(-\Delta_\alpha)^{\frac{1}{2}}\hat \theta, (-\Delta_{\al})^{-\frac{1}{2}}B\hat \theta\rangle\\	
	=&\frac{1}{2} \langle \cos y(-\Delta_\alpha)^{\frac{1}{2}}\hat u, (-\Delta_{\al})^{-\frac{1}{2}}\cos y \hat u\rangle\\
	&+\frac{1}{2} \langle (-\Delta_{\al})^{-\frac{1}{2}}\cos y \hat u,\cos y(-\Delta_\alpha)^{\frac{1}{2}}\hat u\rangle\\
	=&\frac{\pi}{2}\sum_{\beta\in\mathbb Z}\Big(\frac{(\al^2+(\beta-1)^2)^{\frac{1}{2}}}{(\al^2+\beta^2)^{\frac{1}{2}}}|\tilde u(\beta-1)|^2+\frac{(\al^2+(\beta+1)^2)^{\frac{1}{2}}}{(\al^2+\beta^2)^{\frac{1}{2}}}|\tilde u(\beta+1)|^2\Big)\\
	&+\frac{\pi}{2}\sum_{\beta\in\mathbb Z}(\frac{(\al^2+(\beta-1)^2)^{\frac{1}{2}}+(\al^2+(\beta+1)^2)^{\frac{1}{2}}}{(\al^2+\beta^2)^{\frac{1}{2}}})\Re(\tilde u(\beta-1)\overline{\tilde u(\beta+1)})\\
	=&\frac{\pi}{4}\sum_{\beta\in\mathbb Z}(\frac{(\al^2+(\beta-1)^2)^{\frac{1}{2}}+(\al^2+(\beta+1)^2)^{\frac{1}{2}}}{(\al^2+\beta^2)^{\frac{1}{2}}})|\tilde u(\beta-1)+\tilde u(\beta+1)|^2\\
	&-\pi \sum_{\beta\in\mathbb Z}\frac{\beta(|\tilde u(\beta-1)|^2-|\tilde u(\beta+1)|^2)}{(\al^2+\beta^2)^{\frac{1}{2}}((\al^2+(\beta-1)^2)^{\frac{1}{2}}+(\al^2+(\beta+1)^2)^{\frac{1}{2}})}\\
	=&\frac{\pi}{4}\sum_{\beta\in\mathbb Z}(\frac{(\al^2+(\beta-1)^2)^{\frac{1}{2}}+(\al^2+(\beta+1)^2)^{\frac{1}{2}}}{(\al^2+\beta^2)^{\frac{1}{2}}})|\tilde u(\beta-1)+\tilde u(\beta+1)|^2\\
	&-\pi \sum_{\beta\in\mathbb Z}\frac{\beta\Re\Big(\big(\tilde u(\beta-1)+\tilde u(\beta+1)\big)\cdot\overline{\big(\tilde u(\beta-1)-\tilde u(\beta+1)\big)}\Big)}{(\al^2+\beta^2)^{\frac{1}{2}}((\al^2+(\beta-1)^2)^{\frac{1}{2}}+(\al^2+(\beta+1)^2)^{\frac{1}{2}})}\\
  \ge&-\frac{\pi}{4}\sum_{\beta\in\mathbb Z}(\al^2+\beta^2)^{\frac{1}{2}}|\tilde u(\beta-1)-\tilde u(\beta+1)|^2\\
  =&-\frac{1}{2}\|(-\Delta_\al)^{\frac{1}{4}}A\hat \theta\|_{L^2}^2,
\end{align*}
which means that
\begin{align*}
	\Re\langle (-\Delta_{\al})^{-\frac{1}{2}}B(-\Delta_\alpha)^{\frac{1}{2}}\hat \theta, B\hat \theta\rangle+\frac{1}{2}\|(-\Delta_\al)^{\frac{1}{4}}A\hat \theta\|_{L^2}^2\ge0.
\end{align*}
As a conclusion, we arrive at
\begin{align*}
  \frac{d}{dt}\mathcal E_2\le&-2\nu E_0-\nu\|(-\Delta_\al)^{\frac{1}{4}}A\hat \theta\|^2_{L^2}+2c_4\nu\|A\hat \theta\|_{L^2}E_0^{\frac{1}{2}}.
\end{align*}
This complete the proof.
\end{proof}
\section{Enhanced dissipation}
In this section, we determine the coefficients $a_1,a_2,a_3$ in 
\begin{align*}
  \Phi(t)=E_0(t)+a_1\nu^{2}t^2E_1(t)+a_2\nu^{2}t^3\mathcal E_1(t)+a_3\nu^{2}t^4\mathcal E_2(t),
\end{align*}
and study the time evolution of $\Phi(t)$. Precisely, we prove the following lemma. 
% such that
%\begin{align*}
%  \frac{d}{dt}\Phi(t)\le -c\nu^{3}t^4 E_0.
%\end{align*}
%And then give the proof of Theorem \ref{Thm-cri}.
\begin{lemma}\label{lem-decay}
  If we take
  \begin{align*}
  a_1=\frac{|\al|^{-1}}{4c_2(6c_4^2+\frac{12c_1}{c_3})^2},\quad a_2=\frac{|\al|^{-\frac{1}{2}}}{8c_2(6c_4^2+\frac{12c_1}{c_3})^3},\quad a_3=\frac{1}{8c_2(6c_4^2+\frac{12c_1}{c_3})^4},
\end{align*}
then for $\nu\le1$ and $0\le t\le\nu^{-\frac{3}{5}}$, it holds that
\begin{align*}
  \Phi(t)\ge E_0(t),\quad  \frac{d}{dt}\Phi(t)\le -a_3\nu^{3}t^4 E_0(t).
\end{align*}
\end{lemma}
\begin{proof}
  For the same reason to Proposition \ref{pro-energy}, we only prove for the case $\al>1$, for which $\mathcal E_1=-\Re\langle iA\hat \theta, \pa_y\hat \theta\rangle_*$ and $\gamma>0$. The proof for $\al<-1$ is the same. 

With the help of Proposition \ref{pro-energy}, we deduce that
	\begin{align*}
		\frac{d}{dt}\Phi(t)\le& -2\nu E_{\frac{1}{2}}\\
		&+2a_1\nu^2tE_1-2a_1\nu^3t^2E_{\frac{3}{2}}-2a_1\nu^2t^2\gamma\mathcal E_1\\
		&+3a_2\nu^2t^2\mathcal E_1+2a_2\nu^3t^3\|(-\Delta)^{\frac{1}{4}}A\hat \theta\|_{L^2}E_{\frac{3}{2}}^{\frac{1}{2}}+a_2c_2\nu^3t^3E_{\frac{1}{2}}\\
		&\qquad\qquad\qquad\qquad\qquad-a_2\gamma\nu^2t^3\|A\hat \theta\|^2_{L^2}-a_2c_3\gamma\nu^2t^3\sum_{\beta\in\mathbb Z}\frac{2\pi\beta^2}{(\al^2+\beta^2)^{\frac{1}{2}}}|\tilde\psi(\beta)|^2\\
		&+4a_3\nu^2t^3\mathcal E_2-2a_3\nu^3t^4E_0-a_3\nu^3t^4\|(-\Delta)^{\frac{1}{4}}A\hat \theta\|_{L^2}^2+2a_3c_4\nu^3t^4\|A\hat \theta\|_{L^2}E_0^{\frac{1}{2}}\\
		=&-(a_2c_2\nu E_{\frac{1}{2}}-a_2c_2\nu^3t^3E_{\frac{1}{2}})\\
		&-(a_1\nu  E_{\frac{1}{2}}-2a_1\nu^2tE_1+a_1\nu^3t^2E_{\frac{3}{2}})\\
		&-\Big(\frac{(3a_2-2a_1\gamma)^2}{2a_3}\nu  E_{\frac{1}{2}}-(3a_2-2a_1\gamma)\nu^2t^2\mathcal E_1+\frac{1}{2}a_3\nu^3t^4\|(-\Delta)^{\frac{1}{4}}A\hat \theta\|_{L^2}^2\Big)\\
		&-\Big(a_1\nu^3t^2E_{\frac{3}{2}}-2a_2\nu^3t^3\|(-\Delta)^{\frac{1}{4}}A\hat \theta\|_{L^2}E_{\frac{3}{2}}^{\frac{1}{2}}+\frac{1}{2}a_3\nu^3t^4\|(-\Delta)^{\frac{1}{4}}A\hat \theta\|_{L^2}^2\Big)\\
		&-(a_3\nu^3t^4E_0-2a_3c_4\nu^3t^4\|A\hat \theta\|_{L^2}E_0^{\frac{1}{2}}+\frac{1}{2}a_2\gamma\nu^2t^3\|A\hat \theta\|^2_{L^2})\\
		&-\big(\frac{1}{2}a_2\gamma\nu^2t^3\|A\hat \theta\|^2_{L^2}-4a_3\nu^2t^3\mathcal E_2+a_2c_3\gamma\nu^2t^3\sum_{\beta\in\mathbb Z}\frac{2\pi\beta^2}{(\al^2+\beta^2)^{\frac{1}{2}}}|\tilde\psi(\beta)|^2\big)\\
		&-(2-a_2c_2-a_1-\frac{(3a_2-2a_1\gamma)^2}{2a_3})\nu  E_{\frac{1}{2}}\\
		&-a_3\nu^3t^4E_0\\
    \eqdefa&I_1+I_2+I_3+I_4+I_5+I_6+I_7+I_8.
	\end{align*}
Our analysis is based on the above inequality. 

For $\nu\le1$ and $t\in[0,\nu^{-3/5}]$, it holds that $\frac{1}{3}|\al|\le\gamma\le|\al|$ and $\nu^3t^3\le\nu$. By the definition, it is easy to see that
\begin{align*}
  E_1\le E_{\frac{1}{2}}^{\frac{1}{2}}E_{\frac{3}{2}}^{\frac{1}{2}},\quad \mathcal E_1\le E_{\frac{1}{2}}^{\frac{1}{2}}\|(-\Delta)^{\frac{1}{4}}A\hat \theta\|_{L^2}.
\end{align*}
From Lemma \ref{lem-est-me2} and the fact that $c_3\le \frac{1}{2}$, we have
\begin{align*}
  \frac{1}{2}a_2\gamma\nu^2t^3\|A\hat \theta\|^2_{L^2}+a_2c_3\gamma\nu^2t^3\sum_{\beta\in\mathbb Z}\frac{2\pi\beta^2}{(\al^2+\beta^2)^{\frac{1}{2}}}|\tilde\psi(\beta)|^2\ge \frac{c_3}{c_1}|\al|^{-\frac{1}{2}}a_2\gamma\nu^2t^3\mathcal E_2.
\end{align*}
Therefore, to ensure that all $I_i$ for $i=1,\dots,7$ are negative, we obtain the following constraint conditions:
\begin{align*}
	a_2^2\le \frac{1}{2}a_1a_3,\quad a_3 \le \frac{a_2|\al|}{6c_4^2},\quad a_3\le \frac{c_3a_2|\al|^{\frac{1}{2}}}{12c_1},\quad a_2\le \frac{1}{2c_2},\quad a_1^2|\al|^2\le \frac{a_3}{2}.
\end{align*}
So we choose
\begin{align*}
	a_1=\frac{|\al|^{-1}}{4c_2(6c_4^2+\frac{12c_1}{c_3})^2},\quad a_2=\frac{|\al|^{-\frac{1}{2}}}{8c_2(6c_4^2+\frac{12c_1}{c_3})^3},\quad a_3=\frac{1}{8c_2(6c_4^2+\frac{12c_1}{c_3})^4}
\end{align*}
to fulfill the above conditions.

As a result, we have
\begin{align*}
	\frac{d}{dt}\Phi (t)\le-a_3\nu^3t^4E_0(t).
\end{align*}

From Lemma \ref{lem-equivalence}, one can see that
\begin{align*}
  \mathcal E_1\le E_{1}^{\frac{1}{2}}\|A\hat\theta\|_{L^2}\le  E_{1}^{\frac{1}{2}}\mathcal E_{2}^{\frac{1}{2}}.
\end{align*}
As $a_2^2\le \frac{1}{2}a_1a_3$, it follows that $\Phi(t)\ge E_0(t)$. It completes the proof.
\end{proof}

Now we give the proof of Theorem \ref{Thm-cri}.
\begin{proof}[Proof of Theorem \ref{Thm-cri}]
  From Lemma \ref{lem-decay}, one can see that
  \begin{align*}
  \Phi(t)\le\Phi(0)=E_0(0),\quad\forall t\in[0,\nu^{-3/5}].
\end{align*}
Next, we introduce a new energy functional which is defined on $t\in[\nu^{-3/5},\nu^{-1}]$
\begin{align*}
  \tilde\Phi(t)=E_0(t)+a_1\nu^{4/5}E_1(t)+a_2\nu^{1/5}\mathcal E_1(t)+a_3\nu^{-2/5}\mathcal E_2(t).
\end{align*}
It is easy to check that 
\begin{align*}
  \frac{d}{dt}\tilde\Phi(t)=&-2 \nu E_{\frac{1}{2}}-2a_1\nu^{9/5}E_{\frac{3}{2}}-2a_1\nu^{4/5}|\gamma|\mathcal E_1+2a_2\nu^{6/5}\|(-\Delta_\alpha)^{\frac{1}{4}}A\hat \theta\|_{L^2}E_{\frac{3}{2}}^{\frac{1}{2}}\\
  &+a_2c_2\nu^{6/5} E_{\frac{1}{2}}-a_2|\gamma|\nu^{1/5}\|A\hat \theta\|^2_{L^2}-a_2c_3|\gamma|\nu^{1/5}\sum_{\beta\in\mathbb Z}\frac{2\pi\beta^2}{(\al^2+\beta^2)^{\frac{1}{2}}}|\tilde\psi(\beta)|^2\\
  &-2a_3\nu^{3/5} E_0-a_3\nu^{3/5}\|(-\Delta_\al)^{\frac{1}{4}}A\hat \theta\|_{L^2}^2+4a_3c_4\nu^{3/5}\|A\hat \theta\|_{L^2}E_0^{\frac{1}{2}}\\ 
  \le&-a_3\nu^{3/5} E_0-a_1\nu^{7/5}E_1-a_2\nu^{4/5}\mathcal E_1-a_3\nu^{1/5}\mathcal E_2\\
  \le&-a_3\nu^{3/5}\tilde\Phi(t).
\end{align*}
Here $a_1,a_2,a_3$ are given in Lemma \ref{lem-decay}.

It is clear that
\begin{align*}
  \tilde\Phi(\nu^{-3/5})=\Phi(\nu^{-3/5})\le E_0(0).
\end{align*}
Then we have for $t\ge\nu^{-3/5}$ that
\begin{align*}
  E_0(t)\le \tilde\Phi(t)\le e^{-a_3\nu^{3/5}(t-\nu^{-3/5})}\tilde\Phi(\nu^{-3/5})\le2e^{-a_3\nu^{3/5}t}E_0(0).
\end{align*}
As $a_3>0$ depends only on $c_0$, the result of Theorem \ref{Thm-cri} follows immediately.
\end{proof}

\section{Sub-critical QG equation}
In this section, we study the equation \eqref{eq:L-QG-al} with $s=1$
\begin{align}\label{eq-ft-x-sub}
  \pa_t\hat \theta-\nu\Delta_\al\hat \theta+i\al e^{-\nu t}\cos y \big(1-(-\Delta_{\al})^{-\frac{1}{2}}\big)\hat \theta=0,
\end{align}
which is the linearized sub-critical QG equation.

Now let
\begin{align*}
  \mathcal L_{\nu,1}=-\nu\Delta_\al+i\gamma(t)B.
\end{align*}
Then \eqref{eq-ft-x-sub} can be written as
\begin{align*}
  \pa_t\hat \theta+\mathcal L_{\nu,1} \hat \theta=0.
\end{align*}
Here we still use the notations
\begin{align*}
  A=\sin y \big(1-(-\Delta_{\al})^{-\frac{1}{2}}\big),\quad B=\cos y \big(1-(-\Delta_{\al})^{-\frac{1}{2}}\big),\quad \gamma(t)=\alpha e^{-\nu t},
\end{align*}
as well as the inner product
\begin{align*}
  \langle u, w\rangle_*=\langle u, w-(-\Delta_{\al})^{-\frac{1}{2}}w\rangle.
\end{align*}

For the sub-critical case, we introduce the new energy functional with different time weight: 
\begin{align*}
  \Phi_1(t)=E_0+a_1\nu t E_1+a_2\nu t^2\mathcal E_1+a_3\nu t^3\mathcal E_2,
\end{align*}
Where $E_0,E_1,\mathcal E_1,\mathcal E_2$ are consistent with the critical case,
\begin{align*}
	E_0=\|\hat \theta\|_*^2,\quad E_1=\|\pa_y\hat \theta\|_*^2,\quad \mathcal E_1=-\frac{\al}{|\al|}\Re\langle iA\hat \theta, \pa_y\hat \theta\rangle_*,\quad \mathcal E_2=\|\hat \theta\|_*^2-\|B\hat \theta\|_*^2.
\end{align*}
We have the following proposition: 
\begin{proposition}\label{pro-energy-sub}
Suppose that $|\al|\geq c_0>1$. Then it holds that 
\begin{align*}
  \frac{d}{dt}E_0=&-2\nu E_{1}-2\nu|\al|^2 E_{0},\\
  \frac{d}{dt}E_1=&-2\nu E_{2}-2\nu|\al|^2 E_{1}-2|\gamma|\mathcal E_1,\\
  \frac{d}{dt}\mathcal E_1\le&-\nu(2|\al|^2+1)\mathcal E_1+2\nu\|A\pa_y\hat \theta\|_{L^2}E_{2}^{\frac{1}{2}}-|\gamma|\|A\hat \theta\|^2_{L^2}-c_3|\gamma|\sum_{\beta\in\mathbb Z}\frac{2\pi\beta^2}{(\al^2+\beta^2)^{\frac{1}{2}}}|\tilde\psi(\beta)|^2,\\
  \frac{d}{dt}\mathcal E_2\le&-2\nu\Big(E_{\frac{1}{2}}+\|A\pa_y\hat \theta\|^2_{L^2}\Big),
\end{align*}
where 
\begin{align*}
  E_{\frac{1}{2}}=2\pi\sum_{\beta\in\mathbb Z}((\al^2+\beta^2)^{\frac{1}{2}}-1)|\tilde \theta(\beta)|^2,\quad \quad
E_{2}=\|\pa_y^2\hat\theta\|_*^2=2\pi\sum_{\beta\in\mathbb Z} \frac{\beta^4((\al^2+\beta^2)^{\frac{1}{2}}-1)}{(\al^2+\beta^2)^{\frac{1}{2}}}|\tilde \theta(\beta)|^2,
\end{align*}
and $0<c_3\le \frac{1}{2}$ is an independent constant. 
\end{proposition}
\begin{proof}
  We prove the case $\al>1$. The proof for the case $\al<-1$ is the same. 
  
  By using the symmetric property of $B$ and following the idea in the proof of Proposition \ref{pro-energy}, we obtain that 
  \begin{align*}
  \frac{d}{dt}E_0=&-2\Re\langle \mathcal L_{\nu,1}\hat \theta, \hat \theta\rangle_*=2\nu\Re\langle \Delta_\al\hat \theta, \hat \theta\rangle_*-2\gamma\Re\langle iB\hat \theta, \hat \theta\rangle_*=-2\nu E_{1}-2\nu|\al|^2 E_{0},
\end{align*}
and
\begin{align*}
  \frac{d}{dt}E_1=&-2\Re\langle \pa_y\mathcal L_{\nu,1}\hat \theta, \hat \theta\rangle_*=2\nu\Re\langle \Delta_\al\pa_y\hat \theta, \pa_y\hat \theta\rangle_*-2\gamma\Re\langle i\pa_y(B\hat \theta), \pa_y\hat \theta\rangle_*\\
  =&2\nu\Re\langle \Delta_\al\pa_y\hat \theta, \pa_y\hat \theta\rangle_*-2\gamma\Re\langle iB\pa_y\hat \theta, \pa_y\hat \theta\rangle_*+2\gamma\Re\langle iA\hat \theta, \pa_y\hat \theta\rangle_*\\
  =&-2\nu E_{2}-2\nu|\al|^2 E_{1}-2\gamma\mathcal E_1.
\end{align*}

For $\mathcal E_1=-\Re\langle iA\hat \theta, \pa_y\hat \theta\rangle_*$, by using Lemma \ref{lem-com-ba} and the fact that $B$ is symmetric, we have
\begin{align*}
  \frac{d}{dt}\mathcal E_1=&-\nu\Re\langle iA\Delta_\al\hat \theta, \pa_y\hat \theta\rangle_*-\gamma\Re\langle AB\hat \theta, \pa_y\hat \theta\rangle_*\\
  &-\nu\Re\langle iA\hat \theta, \Delta_\al\pa_y\hat \theta\rangle_*+\gamma\Re\langle A\hat \theta, B\pa_y\hat \theta\rangle_*-\gamma\Re\langle A\hat \theta, A\hat \theta\rangle_*\\
  =&-\nu\Re\langle iA\Delta_\al\hat \theta, \pa_y\hat \theta\rangle_*-\nu\Re\langle iA\hat \theta, \Delta_\al\pa_y\hat \theta\rangle_*+\gamma\Re\langle (BA-AB)\hat \theta, \pa_y\hat \theta\rangle_*-\gamma\Re\langle A\hat \theta, A\hat \theta\rangle_*\\
  =&-\nu(2|\al|^2+1)\mathcal E_1 +2\nu\Re\langle iA\pa_y\hat \theta, \pa_y^2\hat \theta\rangle_*+\gamma\Re\langle (BA-AB)\hat \theta, \pa_y\hat \theta\rangle_*-\gamma\Re\langle A\hat \theta, A\hat \theta\rangle_*\\
  \le&-\nu(2|\al|^2+1)\mathcal E_1+2\nu\|A\pa_y\hat \theta\|_*E_{2}^{\frac{1}{2}}-\gamma\|A\hat \theta\|^2_{L^2}-c_3\gamma\sum_{\beta\in\mathbb Z}\frac{2\pi\beta^2}{(\al^2+\beta^2)^{\frac{1}{2}}}|\tilde\psi(\beta)|^2.
\end{align*}

For $\mathcal E_2=\|\hat \theta\|_*^2-\|B\hat \theta\|_*^2$, recalling the identity \eqref{eq-uw}, one can deduce that
\begin{align*}
  \frac{d}{dt}\mathcal E_2=&2\nu\Re\langle \Delta_\al\hat \theta, \hat \theta\rangle_*-2\nu\Re\langle B\Delta_\al\hat \theta, B\hat \theta\rangle_*-2\gamma\Re\langle iB\hat \theta, \hat \theta\rangle_*+2\gamma\Re\langle iBB\hat \theta, B\hat \theta\rangle_*\\
  =&2\nu\Re\langle \Delta_\al\hat \theta, \hat \theta\rangle_*-2\nu\Re\langle B\Delta_\al\hat \theta, B\hat \theta\rangle_*\\
  =&-2\nu\Big(\big\langle (-\Delta_\alpha)^{\frac{1}{2}}\hat \theta, \hat \theta-(-\Delta_{\al})^{-\frac{1}{2}}\hat \theta\big\rangle-\Re\langle (-\Delta_{\al})^{-\frac{1}{2}}B\Delta_\al\hat \theta, B\hat \theta\rangle-\Re\langle A\Delta_\al\hat \theta, A\hat \theta\rangle\Big)\\
  \le&-2\nu\Big(E_{\frac{1}{2}}+\al^2\|A \hat \theta\|_{L^2}^2+\|A\pa_y\hat \theta\|^2_{L^2}\Big).
\end{align*}
Here we use the fact that
\begin{align*}
  &\Re\langle A\Delta_\al\hat \theta, A\hat \theta\rangle+\Re\langle (-\Delta_{\al})^{-\frac{1}{2}}B\Delta_\al\hat \theta, B\hat \theta\rangle\\
  =&-\al^2 \|\sin y \hat u\|_{L^2}^2+\frac{1}{2}\int_\bbT \sin^2y (\hat u\cdot\overline{\pa_y^2\hat u}+\pa_y^2\hat u\cdot\overline{\hat u})dy\\
  &-\Re\langle (-\Delta_{\al})^{\frac{1}{2}}\cos y \hat u, \cos y\hat u\rangle+\Re\langle (-\Delta_{\al})^{-\frac{1}{2}}[\cos y,\pa_y^2]\hat u, \cos y\hat u\rangle\\
  =&\|\cos y \hat u\|_{L^2}^2-(\al^2+1)\|\sin y \hat u\|_{L^2}^2-\|\sin y \pa_y\hat u\|_{L^2}^2\\
  &-\langle (-\Delta_{\al})^{\frac{1}{2}}\cos y\hat u, \cos y\hat u\rangle+2\Re\langle (-\Delta_{\al})^{-\frac{1}{2}}\pa_y(\sin y \hat u), \cos y\hat u\rangle-\langle (-\Delta_{\al})^{-\frac{1}{2}}\cos y\hat u, \cos y\hat u\rangle\\
  \le&-(\al^2+1)\|\sin y \hat u\|_{L^2}^2+2\Re\langle (-\Delta_{\al})^{-\frac{1}{2}}\pa_y(\sin y \hat u), \cos y\hat u\rangle-\|\cos y \hat u\|_{L^2}^2-\|\sin y \pa_y\hat u\|_{L^2}^2\\
  \le&-\al^2\|\sin y \hat u\|_{L^2}^2-\|\sin y \pa_y\hat u\|_{L^2}^2=-\al^2\|A \hat \theta\|_{L^2}^2-\|A\pa_y\hat \theta\|^2_{L^2},
\end{align*}
where $\hat u=(1-(-\Delta_{\al})^{-\frac{1}{2}})\hat \theta$.

This finishes the proof.
\end{proof}

Now we give the proof of Theorem \ref{thm-sub}.
\begin{proof}[Proof of Theorem \ref{thm-sub}]
  Recall the energy functional
  \begin{align*}
  \Phi_1(t)=E_0+a_1\nu t E_1+a_2\nu t^2\mathcal E_1+a_3\nu t^3\mathcal E_2.
\end{align*}
By Proposition \ref{pro-energy-sub}, we get that
\begin{equation}\label{ineq-sub}
  \begin{aligned}    
    \frac{d}{dt}\Phi_1(t)\le& -2\nu E_{1} -2\nu |\al|^2E_{0}\\
    &+a_1\nu E_1-2a_1\nu^2tE_{2}-2a_1|\al|^2\nu^2tE_{1}-2a_1\nu t|\gamma|\mathcal E_1\\
    &+2a_2\nu t \mathcal E_1-a_2\nu^2t^2(2|\al|^2+1)\mathcal E_1+2a_2\nu^2t^2\|A\pa_y\hat \theta\|_*E_{2}^{\frac{1}{2}}\\
    &\qquad\qquad\qquad\qquad\qquad\qquad\quad-a_2|\gamma|\nu t^2\|A\hat \theta\|^2_{L^2}-a_2c_3|\gamma|\nu t^2\sum_{\beta\in\mathbb Z}\frac{2\pi\beta^2|\tilde\psi(\beta)|^2}{(\al^2+\beta^2)^{\frac{1}{2}}}\\
    &+3a_3\nu t^2\mathcal E_2-2a_3\nu^2t^3E_{\frac{1}{2}}-2a_3\nu^2t^3|\al|^2 \|A \hat \theta\|^2_{L^2}-2a_3\nu^2t^3 \|A\pa_y \hat \theta\|^2_{L^2}.    
  \end{aligned}
\end{equation}
It is clear that
\begin{align*}
  \mathcal E_1\le \|A \hat \theta\|_{L^2}E_1^{\frac{1}{2}}.
\end{align*}
By using Lemma \ref{lem-est-me2}, we have
\begin{align*}
  \frac{1}{2}a_2|\gamma|\nu t^2\|A\hat \theta\|^2_{L^2}+a_2c_3|\gamma|\nu t^2\sum_{\beta\in\mathbb Z}\frac{2\pi\beta^2|\tilde\psi(\beta)|^2}{(\al^2+\beta^2)^{\frac{1}{2}}}\ge \frac{c_3}{c_1}|\al|^{-\frac{1}{2}}a_2\gamma\nu t^2\mathcal E_2.
\end{align*}
Lemma \ref{lem-equivalence} gives us that 
\begin{align*}
  \nu^{\frac{3}{2}}t^{\frac{5}{2}}E_0\le(\nu^2t^3E_{\frac{1}{2}}^{\frac{1}{2}})\Big(\nu t^2\big(\sum_{\beta\in\mathbb Z}((\al^2+\beta^2)^{\frac{1}{2}}-1)|\tilde\psi(\beta)|^2\big)^{\frac{1}{2}}\Big)\le(\nu^2t^3E_{\frac{1}{2}}^{\frac{1}{2}})(\nu t^2\mathcal E_2^{\frac{1}{2}}).
\end{align*}
Therefore, we can choose
\begin{align*}
  a_1=\frac{|\al|^{-1}}{\left(\frac{18 c_1}{c_3}\right)^2} ,\quad a_2=\frac{|\al|^{-\frac{1}{2}}}{2\left(\frac{18 c_1}{c_3}\right)^3} ,\quad a_3=\frac{1}{2\left(\frac{18 c_1}{c_3}\right)^4} 
\end{align*}
to ensure for $t\in[0,\nu^{-\frac{3}{7}}]$ that
  \begin{align*}
    \Phi_1(t)\ge E_0(t),\quad \frac{d}{dt}\Phi_1(t)\le&-a_3\nu^{\frac{3}{2}}t^{\frac{5}{2}}E_0(t).
  \end{align*}
Similarly, we introduce a new energy functional defined on $t\in[\nu^{-\frac{3}{7}},\nu^{-1}]$,
\begin{align*}
  \tilde\Phi_1(t)=E_0(t)+a_1\nu^{\frac{4}{7}}E_1(t)+a_2\nu^{\frac{1}{7}}\mathcal E_1(t)+a_3\nu^{-\frac{2}{7}}\mathcal E_2(t),
\end{align*}
which satisfies
  \begin{align*}
    \frac{d}{dt}\tilde\Phi_1(t)\le&-a_3\nu^{\frac{3}{7}}\tilde\Phi_1(t),
  \end{align*}
and
\begin{align*}
  \tilde\Phi_1(\nu^{-\frac{3}{7}})=\Phi_1(\nu^{-\frac{3}{7}})\le E_0(0).
\end{align*}
Therefore, we conclude for $t\in[\nu^{-\frac{3}{7}},\nu^{-1}]$ that
\begin{align*}
  E_0(t)\le \tilde\Phi_1(t)\le e^{-a_3\nu^{\frac{3}{7}}(t-\nu^{-\frac{3}{7}})}\tilde\Phi_1(\nu^{-\frac{3}{7}})\le2e^{-a_3\nu^{\frac{3}{7}}t}E_0(0).
\end{align*}  
The assertion of Theorem \ref{thm-sub} follows immediately.
\end{proof}

\section{Transport fractional diffusion equation}
In this section, we consider the toy model: 
\begin{align*}
	\pa_t\hat \theta(t,\al,y)+\mathcal{L}^S_{\al,\nu,\frac{1}{2}}\hat \theta(t,\al,y)=0,
\end{align*}
where
\begin{align*}
	\mathcal{L}^S_{\al,\nu,\frac{1}{2}}=\nu (-\Delta_{\al})^{\frac{1}{2}}+i\al \cos y.
\end{align*}
%Thus, we have
%\begin{align*}
%	\pa_t\hat \theta(t,\al,y)+\nu(-\Delta_{\al})^{\frac{1}{2}}\hat \theta(t,\al,y)+i \al\cos y\hat \theta(t,\al,y)=0.
%\end{align*}
We define the energy functional
\begin{align*}
	\Phi^S(t)=E_0^S(t)+a_1\nu^{2}t^2E_1^S(t)+a_2\nu^{2}t^3\mathcal E_1^S(t)+a_3\nu^{2}t^4\mathcal E_2^S(t),
\end{align*}
where
\begin{align*}
	E_0^S=\|\hat \theta\|_{L^2}^2,\quad E_1^S=\|\pa_y\hat \theta\|_{L^2}^2,\quad \mathcal E_1^S=-\frac{\al}{|\al|}\langle i\sin y\hat \theta,\pa_y\hat \theta\rangle, \quad\mathcal E_2^S=\|\sin y\hat \theta\|_{L^2}^2,
\end{align*}
%We do not use the energy functional
%\begin{align*}
%	\tilde\Phi(t)=E_0(t)+a_1\nu^{2/3}E_1(t)+a_2\mathcal E_1(t)+a_3\nu^{-2/3}\mathcal E_2(t),
%\end{align*}
%because as $\tilde\nu\to 0$, $a_3\nu^{-2/3}\mathcal E_2(t)$ may be a really big term. As a result, it is hard to get
%\begin{align*}
%	E_0(t)\le e^{-\nu^{2/3}t}E_0(0)
%\end{align*}
%from
%\begin{align*}
%	\tilde\Phi(t)\le e^{-\nu^{2/3}t}\tilde\Phi(0).
%\end{align*}
Let us first introduce a new interpolation lemma.
\begin{lemma}\label{lem-inter}
  It holds that
  \begin{align*}
    \|\hat\theta\|^2_{L^2}\le2\|\pa_y(\cos y \hat \theta)\|_{L^2}\|\sin y\hat\theta\|_{L^2}.
  \end{align*}
\end{lemma}
\begin{proof}
A direct calculation shows that,
  \begin{align*}
  &-\int_\bbT\sin y\hat \theta\overline{\pa_y(\cos y \hat \theta)}+\pa_y(\cos y \hat \theta)\overline{\sin y\hat \theta}dy\\
  =&2\int_\bbT\sin^2 y|\hat \theta|^2-\sin y\cos y\hat \theta\overline{\pa_y\hat \theta}-\sin y\cos y\pa_y\hat \theta\overline{\hat \theta}dy\\
  =&2\int_\bbT\sin^2 y|\hat \theta|^2dy-\frac{1}{2}\int_\bbT\sin 2y\pa_y|\hat \theta|^2dy\\
  =&2\int_\bbT\sin^2 y|\hat \theta|^2dy+\int_\bbT\cos 2y|\hat \theta|^2dy\\
  =&\int_\bbT|\hat \theta|^2dy.
\end{align*}
The results follow immediately.
\end{proof}
Now we give the proof for Theorem \ref{thm-toy}.
\begin{proof}[Proof of Theorem \ref{thm-toy}]
  We start with some basic estimates, here we still focus only on the case $\al>1$. It is easy to check that
\begin{align*}
  \frac{d}{dt}E_0^S=-2\nu\langle (-\Delta_\alpha)^{\frac{1}{2}}\hat \theta, \hat \theta\rangle-2 \al\Re\langle i\cos y\hat \theta, \hat \theta\rangle=-2\nu E_{\frac{1}{2}}^S,
\end{align*}
and 
\begin{align*}
  \frac{d}{dt}E_1^S=&-2\nu\langle (-\Delta_\alpha)^{\frac{1}{2}}\pa_y\hat \theta, \pa_y\hat \theta\rangle-2\al\Re\langle i\cos y\pa_y\hat \theta,\pa_y \hat \theta\rangle+2\al\Re\langle i\sin y \hat \theta,\pa_y \hat \theta\rangle\\
  =&-2\nu E_{\frac{3}{2}}^S-2\al\mathcal E_1^S.
\end{align*}
Here we use $E_{\frac{1}{2}}^S$ and $E_{\frac{3}{2}}^S$ to denote $\langle (-\Delta_\al)^{\frac{1}{4}}\hat \theta, (-\Delta_\al)^{\frac{1}{4}}\hat \theta\rangle$ and $\langle (-\Delta_\al)^{\frac{1}{4}}\pa_y\hat \theta, (-\Delta_\al)^{\frac{1}{4}}\pa_y\hat \theta\rangle$ respectively.

For $\mathcal E_1^S=-\langle i\sin y\hat \theta,\pa_y\hat \theta\rangle$, we have
\begin{align*}
  \frac{d}{dt}\mathcal E_1^S=&\nu\Re\langle i\sin y(-\Delta_\alpha)^{\frac{1}{2}}\hat \theta, \pa_y\hat \theta\rangle-\al\Re\langle \sin y\cos y\hat \theta, \pa_y\hat \theta\rangle\\
  &+\nu\Re\langle i\sin y\hat \theta, (-\Delta_\alpha)^{\frac{1}{2}}\pa_y\hat \theta\rangle+\al\Re\langle \sin y\hat \theta, \cos y\pa_y\hat \theta\rangle-\al\Re\langle \sin y\hat \theta, \sin y\hat \theta\rangle\\
  =&2\nu\Re\langle i(-\Delta_\al)^{\frac{1}{4}}\sin y\hat \theta, (-\Delta_\al)^{\frac{1}{4}}\pa_y\hat \theta\rangle-\al\langle \sin y\hat \theta, \sin y\hat \theta\rangle+\nu\Re\langle i[\sin y,(-\Delta_\alpha)^{\frac{1}{2}}]\hat \theta, \pa_y\hat \theta\rangle.
\end{align*}
For the last term $\nu\Re\langle i[\sin y,(-\Delta_\alpha)^{\frac{1}{2}}]\hat \theta, \pa_y\hat \theta\rangle$, it holds that
\begin{align*}
  &\big|\Re\langle i[\sin y,(-\Delta_\alpha)^{\frac{1}{2}}]\hat \theta, \pa_y\hat \theta\rangle\big|\\
  =&\frac{1}{2}\Big|\int_\bbT i\Big(\sin y(-\Delta_\alpha)^{\frac{1}{2}}-(-\Delta_\alpha)^{\frac{1}{2}}\sin y\Big)\hat \theta\cdot\overline{\pa_y\hat \theta}dy\\
  &\qquad\qquad+\frac{1}{2}\int_\bbT \pa_y\hat \theta\cdot\overline{i\Big(\sin y(-\Delta_\alpha)^{\frac{1}{2}}-(-\Delta_\alpha)^{\frac{1}{2}}\sin y\Big)\hat \theta}dy\Big|\\
  =&\frac{\pi}{2}\Big|\sum_{\beta\in\mathbb Z}i\beta\Big((\al^2+(\beta-1)^2)^{\frac{1}{2}}-(\al^2+\beta^2)^{\frac{1}{2}}\Big)\big(\tilde \theta(\beta)\overline{\tilde \theta(\beta-1)}-\tilde \theta(\beta-1)\overline{\tilde \theta(\beta)}\big)\\
  &\qquad\qquad-\sum_{\beta\in\mathbb Z}i\beta\Big((\al^2+(\beta+1)^2)^{\frac{1}{2}}-(\al^2+\beta^2)^{\frac{1}{2}}\Big)\big(\tilde \theta(\beta)\overline{\tilde \theta(\beta+1)}-\tilde \theta(\beta+1)\overline{\tilde \theta(\beta)}\big)\Big|\\
  =&\frac{\pi}{2}\Big|\sum_{\beta\in\mathbb Z}i\beta \frac{1-2\beta}{(\al^2+(\beta-1)^2)^{\frac{1}{2}}+(\al^2+\beta^2)^{\frac{1}{2}}}\big(\tilde \theta(\beta)\overline{\tilde \theta(\beta-1)}-\tilde \theta(\beta-1)\overline{\tilde \theta(\beta)}\big)\\
  &-\sum_{\beta\in\mathbb Z}i\beta\frac{1+2\beta}{(\al^2+(\beta+1)^2)^{\frac{1}{2}}+(\al^2+\beta^2)^{\frac{1}{2}}}\big(\tilde \theta(\beta)\overline{\tilde \theta(\beta+1)}-\tilde \theta(\beta+1)\overline{\tilde \theta(\beta)}\big)\Big|\\
  \le&4E_{\frac{1}{2}}^S.
\end{align*}
Therefore, we get
\begin{align*}
  \frac{d}{dt}\mathcal E_1^S\le&2\nu\|(-\Delta_\al)^{\frac{1}{4}}\sin y\hat \theta\|_{L^2}E_{\frac{3}{2}}^{S \frac{1}{2}}+4\nu E_{\frac{1}{2}}^S-\al\mathcal E_2^S.
\end{align*}

For $\mathcal E_2^S=\|\sin y\hat \theta\|_{L^2}^2$, we have
\begin{align*}
  \frac{d}{dt}\mathcal E_2^S=&-2\nu\Re\langle \sin y(-\Delta_\alpha)^{\frac{1}{2}}\hat \theta, \sin y\hat \theta\rangle-2\al\Re\langle i \sin y\cos y\hat \theta, \sin y\hat \theta\rangle\\
  =&-2\nu\Re\langle \sin y(-\Delta_\alpha)^{\frac{1}{2}}\hat \theta, \sin y\hat \theta\rangle\\
  =&-2\nu\Re\langle (-\Delta_\al)^{\frac{1}{4}}\sin y\hat \theta,(-\Delta_\al)^{\frac{1}{4}} \sin y\hat \theta\rangle-2\nu\Re\langle[\sin y,(-\Delta_\alpha)^{\frac{1}{2}}]\hat \theta, \sin y\hat \theta\rangle\\
  \le&-2\nu\| (-\Delta_\al)^{\frac{1}{4}}\sin y\hat \theta\|_{L^2}^2+4\nu E_0^{S \frac{1}{2}}\mathcal E_2^{S \frac{1}{2}},
\end{align*}
where we use the fact that
\begin{align*}
  \left|\Re\langle [\sin y,(-\Delta_\alpha)^{\frac{1}{2}}]\hat \theta, \sin y\hat \theta\rangle\right|\le\|[\sin y,(-\Delta_\alpha)^{\frac{1}{2}}]\hat \theta\|_{L^2}\mathcal E_2^{S \frac{1}{2}},
\end{align*}
and
\begin{align*}
  &\|[\sin y,(-\Delta_\alpha)^{\frac{1}{2}}]\hat \theta\|_{L^2}^2\\
  =&\int_\bbT \Big(\sin y(-\Delta_\alpha)^{\frac{1}{2}}-(-\Delta_\alpha)^{\frac{1}{2}}\sin y\Big)\hat \theta\cdot\overline{\Big(\sin y(-\Delta_\alpha)^{\frac{1}{2}}-(-\Delta_\alpha)^{\frac{1}{2}}\sin y\Big)\hat \theta}dy\\
  =&\frac{\pi}{2}\sum_{\beta\in\mathbb Z} \Big|\frac{(1-2\beta)\tilde \theta(\beta-1)}{(\al^2+(\beta-1)^2)^{\frac{1}{2}}+(\al^2+\beta^2)^{\frac{1}{2}}}-\frac{(1+2\beta)\tilde \theta(\beta+1)}{(\al^2+(\beta+1)^2)^{\frac{1}{2}}+(\al^2+\beta^2)^{\frac{1}{2}}}\Big|^2\\
  \le&4E_0^S.
\end{align*}

As a result, we have
\begin{align*}
  \frac{d}{dt}\Phi^S (t)\le&-2 \nu E_{\frac{1}{2}}^S+2a_1\nu^2t E_1^S-2a_1\nu^3t^2E_{\frac{3}{2}}^S-2a_1\nu^2t^2\al\mathcal E_1^S\\
  &+3a_2\nu^2t^2\mathcal E_1^S+2a_2\nu^3t^3 \|(-\Delta_\al)^{\frac{1}{4}}\sin y\hat \theta\|_{L^2}E_{\frac{3}{2}}^{S \frac{1}{2}}+4a_2\nu^3t^3 E_{\frac{1}{2}}^S-a_2\nu^2t^3\al\mathcal E_2^S\\
  &+4a_3\nu^2t^3\mathcal E_2^S-2a_3\nu^3t^4\|(-\Delta)^{\frac{1}{4}}\sin y\hat \theta\|_{L^2}^2+4a_3\nu^3t^4E_0^{S \frac{1}{2}}\mathcal E_2^{S \frac{1}{2}}.
\end{align*}

It is clear that 
\begin{align*}
  \nu^2t^2\mathcal E_1^S\le\nu E_{\frac{1}{2}}^S+\nu^3t^4\|(-\Delta_\al)^{\frac{1}{4}}\sin y\hat \theta\|_{L^2}^2,\quad\nu^2t E_1^S\le\nu E_{\frac{1}{2}}^S+\nu^3t^2E_{\frac{3}{2}}^S.
\end{align*}

Therefore, by using Lemma \ref{lem-inter}, we deduce that
\begin{align}\label{eq:001}
\begin{split}
  \nu^{2}t^2E_0^S\le& 2\nu^{2}t^2\mathcal E_2^S+|\al|^{\frac{1}{2}}\nu^{2}t^3\mathcal E_2^S+|\al|^{-\frac{1}{2}}\nu^{2}t E_1^S\\
  \le&2\nu^{2}t^2\mathcal E_2^S+|\al|^{\frac{1}{2}}\nu^{2}t^3\mathcal E_2^S+\nu  E_{\frac{1}{2}}^S+|\al|^{-1}\nu^{3}t^2E_{\frac{3}{2}}^S\\
  \le&2|\al|^{\frac{1}{2}}\nu^{2}t^3\mathcal E_2^S+2\nu  E_{\frac{1}{2}}^S+|\al|^{-1}\nu^{3}t^2E_{\frac{3}{2}}^S.
\end{split}
\end{align}
In the last inequality, we use the following facts. When $t\ge 2$, then
\begin{align*}
  2\nu^{2}t^2\mathcal E_2^S\le\nu^{2}t^3\mathcal E_2^S,
\end{align*}
when $t\le 2$ and $\nu\le \frac{1}{8}$, then
\begin{align*}
  2\nu^{2}t^2\mathcal E_2^S\le\nu  E_{\frac{1}{2}}^S.
\end{align*}

By using the same technique, for $\nu\le \frac{1}{8}$ and $t\le\nu^{-2/3}$, we have
\begin{align*}
  &4a_3\nu^3t^4E_0^{S \frac{1}{2}}\mathcal E_2^{S \frac{1}{2}}\le4a_3\nu^2t^{2/5}E_0^{S \frac{1}{2}}\mathcal E_2^{S \frac{1}{2}}\\
  \le&2a_3\nu^2t^2E_0^S+2a_3\nu^2t^3\mathcal E_2^S\\
  \le&4a_3\nu^2t^2\mathcal E_2^S+2a_3|\al|^{-\frac{1}{2}}\nu^2t E_1^S+4a_3|\al|^{\frac{1}{2}}\nu^2t^3\mathcal E_2^S\\
  \le&4a_3\nu^2t^2\mathcal E_2^S+a_3\nu E_{\frac{1}{2}}^S+a_3|\al|^{-1}\nu^3t^2E_{\frac{3}{2}}^S+4a_3|\al|^{\frac{1}{2}}\nu^2t^3\mathcal E_2^S\\
  \le&2a_3\nu E_{\frac{1}{2}}^S+a_3|\al|^{-1}\nu^3t^2E_{\frac{3}{2}}^S+8a_3|\al|^{\frac{1}{2}}\nu^2t^3\mathcal E_2^S\\
\end{align*}

By choosing
\begin{align*}
a_1=2^{-20}|\alpha|^{-1},\quad a_2=2^{-32}|\alpha|^{-\frac{1}{2}},\quad a_3=2^{-40},
\end{align*}
we deduce for $t\in[1,\nu^{-\frac{2}{3}}]$ that
\begin{align*}
  \Phi^S (t)\ge E^S_0,
\end{align*}
and by \eqref{eq:001},
\begin{align*}
  \frac{d}{dt}\Phi^S (t)\le
  &-2^{-34}|\alpha|^{\frac{1}{2}}\nu^2t^3\mathcal E_2^S-2^{-2}\nu E_{\frac{1}{2}}^S-2^{-22}|\alpha|^{-1}\nu^3t^2E_{\frac{3}{2}}^S\\
  \le&-2^{-35}\nu^2t^2E_0^S.
\end{align*}

Similar to the proof of Theorem \ref{Thm-cri}, we define on $t\in[\nu^{-2/3},\nu^{-1}]$ that
\begin{align*}
  \tilde\Phi^S(t)=E_0^S+2^{-20}|\alpha|^{-1}\nu^{2/3}E_1^S+2^{-32}|\alpha|^{-\frac{1}{2}}\mathcal E_1^S+2^{-40}\nu^{-2/3}\mathcal E_2^S,
\end{align*}
which satisfies
\begin{align*}
  \frac{d}{dt}\tilde\Phi^S(t)\le&-\nu E_{\frac{1}{2}}^S-2^{-19}|\al|^{-1}\nu^{5/3}E_{\frac{3}{2}}^S-2^{-33}|\al|^{\frac{1}{2}}\mathcal E_2^S\le-2^{-35}\nu^{2/3}\tilde\Phi^S(t),
\end{align*}
and
\begin{align*}
  \tilde\Phi^S(\nu^{-2/3})=\Phi^S(\nu^{-2/3})\le E_0^S(0).
\end{align*}
Then it follows that for $t\ge\nu^{-2/3}$
\begin{align*}
  E_0^S(t)\le \tilde\Phi^S(t)\le e^{-2^{-35}\nu^{2/3}(t-\nu^{-2/3})}\tilde\Phi^S(\nu^{-2/3})\le2e^{-2^{-35}\nu^{2/3}t}E_0^S(0).
\end{align*}
The result of Theorem \ref{thm-toy} follows immediately.
\end{proof}

\bibliographystyle{siam.bst} 
\bibliography{references.bib}
\end{document}